\documentclass{amsart}
\usepackage{graphicx}
\usepackage{amssymb}
\usepackage{amscd}
\usepackage{cite}
\usepackage[all]{xy}
\usepackage{url}
\newtheorem{tw}{Theorem}[section]
\newtheorem{prop}[tw]{Proposition}
\newtheorem{lem}[tw]{Lemma}

\theoremstyle{remark}
\newtheorem{uw}[tw]{Remark}
\newtheorem{ex}[tw]{Example}

\theoremstyle{definition}

\newcommand{\cal}[1]{\mathcal{#1}}
\newcommand{\bez}{\setminus}
\newcommand{\podz}{\subseteq}

\newcommand{\fal}[1]{\widetilde{#1}}

\newcommand{\kre}[1]{\overline{#1}}

\newcommand{\gen}[1]{\langle #1 \rangle}
\newcommand{\map}[3]{#1\colon #2\to #3}

\newcommand{\fig}[2]{\includegraphics[width=#1\textwidth]{#2}}

\newcommand{\Mob}{M\"{o}bius strip}

\begin{document}
\numberwithin{equation}{section}

\title[Subgroups generated by two Dehn twists \ldots]
{Subgroups generated by two Dehn twists on a nonorientable surface}

\author{Micha\l\ Stukow}

\thanks{Supported by grants 2012/05/B/ST1/02171 and 2015/17/B/ST1/03235 of National Science Centre, Poland.}
\address[]{
Institute of Mathematics, University of Gda\'nsk, Wita Stwosza 57, 80-952 Gda\'nsk, Poland }

\email{trojkat@mat.ug.edu.pl}



\begin{abstract}
Let $a$ and $b$ be two simple closed curves on an orientable surface $S$ such that their geometric 
intersection number is greater than 1. The group generated 
by corresponding Dehn twists $t_a$ and $t_b$ is known to be isomorphic to the free group of rank 2. In this paper we extend
this result to the case of a nonorientable surface.
\end{abstract}

\maketitle
 \section{Introduction}%
Let $N$ be a smooth, nonorientable, compact surface. We will mainly focus on the local properties of $N$,
hence we allow $N$ to have some boundary components and/or punctures. Let ${\cal H}(N)$ be the group of all
diffeomorphisms $\map{h}{N}{N}$ such that $h$ is the identity on each boundary component and $h$ fixes the set 
of punctures (setwise). By ${\cal M}(N)$ we denote the quotient group of ${\cal H}(N)$ by the subgroup 
that comprises 
the maps isotopic to the identity with an isotopy which fixes the boundary pointwise. ${\cal M}(N)$ is 
known as the
\emph{mapping class group} of $N$. The mapping class group ${\cal M}(S)$ of an orientable surface $S$ 
is defined analogously, but
we consider only orientation preserving maps. Usually, we will use the same letter to denote 
a map and its isotopy class.

Important elements of the mapping class group ${\cal M}(S)$ are \emph{Dehn twists}. 
Dehn twists generate ${\cal M}(S)$, thus obtaining a 
good understanding of possible relations
between them is important. One of the 
basic results in this direction is the following theorem:
\begin{tw}[Ishida \cite{Ishida}]
 If $a$ and $b$ are simple closed curves on an orientable surface $S$ such that the geometric 
 intersection number of $a$ and $b$ is greater than 1, then the group generated by Dehn twists $t_a$ and $t_b$ is free of
 rank 2.
\end{tw}
The main goal of this paper is to extend the above result to the case of a nonorientable surface: 
see Theorem \ref{Main:thrm}. 
Let us mention that 
Dan Margalit observed that if we lift the statement of Theorem \ref{Main:thrm} to the oriented double cover $S$, then we obtain some special cases of the well-known conjecture \cite{MargLein} that two elements of the Torelli subgroup of $S$ either commute or generate a free group. For more details about this correspondence see \cite{StukowMargalit}.

The paper is organized as follows. In Section \ref{sec:Pre}, we establish some basic notation. Section 
\ref{sec:ex} contains some examples that show how the nonorientable case differs from the orientable one.
In Section \ref{sec:joinable} we recall some language introduced in \cite{Stukow_twist}, namely the notion 
of adjacent and joinable segments. Sections \ref{sec:curves}, \ref{sec:Rigid} and \ref{sec:weak} are devoted to the study of properties of curves in 
the neighborhood of $a\cup b$. The main theorem of the paper (Theorem~\ref{Main:thrm}) is proved in 
Section \ref{sec:thm}. This proof is based on five propositions (Propositions \ref{Prop:main:g3}, \ref{Main:prop:s1}, \ref{Main:prop:s3}, \ref{Main:prop:i3} and \ref{Prop:gen2:main}) that are proved in Sections \ref{sec:3} and \ref{sec:specA} to \ref{sec:2}.
 \section{Preliminaries}\label{sec:Pre}
By a \emph{circle} on $N$ we mean an oriented simple closed curve that is disjoint from the boundary of $N$. Usually,
we identify a circle with its image. If two circles $a$ and $b$ intersect, then we always assume that they intersect 
transversely. According to whether a regular neighborhood of a circle is an annulus or a \Mob, we call the 
circle \emph{two-sided} or \emph{one-sided}, respectively. 

We say that a circle is \emph{generic} if
it bounds neither a disk with fewer than two punctures nor a \Mob\ without punctures. It is known
(Corollary 4.5 of \cite{Stukow_twist}) that if $N$ is not a closed Klein bottle, then the circle $a$ is generic
if and only if $t_a$ has infinite order in ${\cal M}(N)$.

For any two circles $a$ and $b$ we define their \emph{geometric intersection number} as follows:
\[I(a,b)={\textrm{inf}}\{|a'\cap b|\,:\, \text{$a'$ is isotopic to $a$}\}.\]
We say that circles $a$ and $b$ \emph{form a bigon} if a disk exists whose boundary is the union of an arc 
of $a$ and an arc of $b$. The following proposition provides a 
useful tool for checking if two circles
are in a minimal position (with respect to $|a\cap b|$).
\begin{prop}[Epstein \cite{Epstein}]
 Let $a$ and $b$ be generic circles on $N$. Then $|a\cap b|=I(a,b)$ if and only if $a$ and $b$ do not form a 
 bigon. \qed
\end{prop}
 \section{Disappointing examples}\label{sec:ex}%
Let $a$ and $b$ be two circles in an oriented surface $S$ such that $I(a,b)\geq2$. The key observation that 
leads to the conclusion that Dehn twists $t_a$ and $t_b$ generate a free group is the following lemma:
\begin{lem}[Lemma 2.3 of \cite{Ishida}]\label{Lem:Oriented}
Assume that circles $a,b,c\subset S$ satisfy $I(a,b)\geq 2$. Then for any nonzero integer $k$ 
 \[I(c,a)>I(c,b)\quad\Longrightarrow \quad I(t_a^k(c),a)<I(t_a^k(c),b).\qed\]
\end{lem}
The above lemma allows to apply the so-called 'ping-pong lemma' (Lemma \ref{PingPong}) and easily conclude that 
$\gen{t_a,t_b}$ is a free group. 

Relations between Dehn twists and geometric intersection numbers are known to become more
complicated if we allow the surface to be nonorientable. Some results in this direction were obtained in 
\cite{Stukow_twist}, but they were too weak to prove a nonorientable version of the above lemma. The main goal
of this section is to show that there is a 
reason for this condition, namely, Lemma \ref{Lem:Oriented} is not 
true on nonorientable surfaces. Moreover, 
finding 
general families of counterexamples is possible. 
Hence, 
no easy fix seems to exist for this situation (for a nontrivial fix, see 
Propositions \ref{Prop:main:g3}, \ref{Main:prop:s1}, \ref{Main:prop:s3}, \ref{Main:prop:i3}, and \ref{Prop:gen2:main}).
\begin{ex}\label{ex:ex1}
 Let $a,b,c$ be two-sided circles indicated in Figure \ref{xr01} (shaded disks are crosscaps, that is, the 
 interiors are to be removed and the boundary points are to be identified by the antipodal map).
 \begin{figure}[h]
\begin{center}
\fig{0.7}{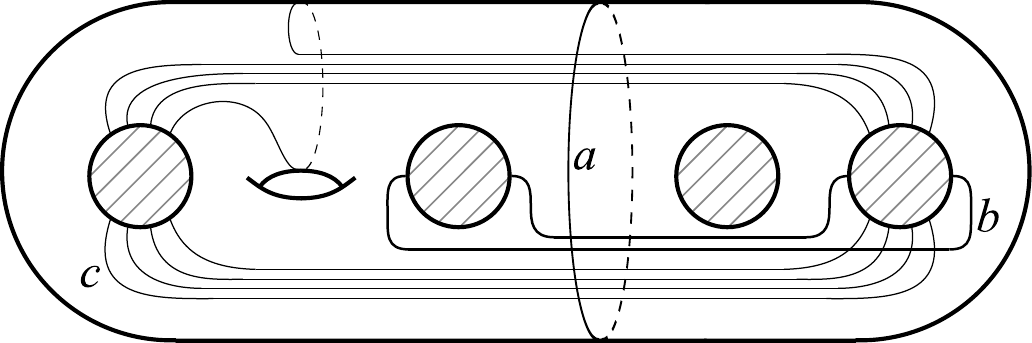}
\caption{Circles $a,b,c$ -- Example \ref{ex:ex1}.}\label{xr01} %
\end{center}
\end{figure}
In particular $I(a,b)=2$ and $I(c,a)=8>I(c,b)=4$. However, 
checking the following is straightforward:
\[I(t_a(c),a)=8>I(t_a(c),b)=4.\]
The above example can be generalized in the obvious way (by changing $c$) to the example where 
$I(a,b)=2$, $I(c,a)=2n>I(c,b)=n$ and $I(t_a(c),a)=2n>I(t_a(c),b)=n$, where $n\geq 1$.
\end{ex}
\begin{ex}\label{ex:ex2}
 Let $a,b,c$ be two-sided circles indicated in Figure \ref{xr03}. 
 \begin{figure}[h]
\begin{center}
\fig{0.6}{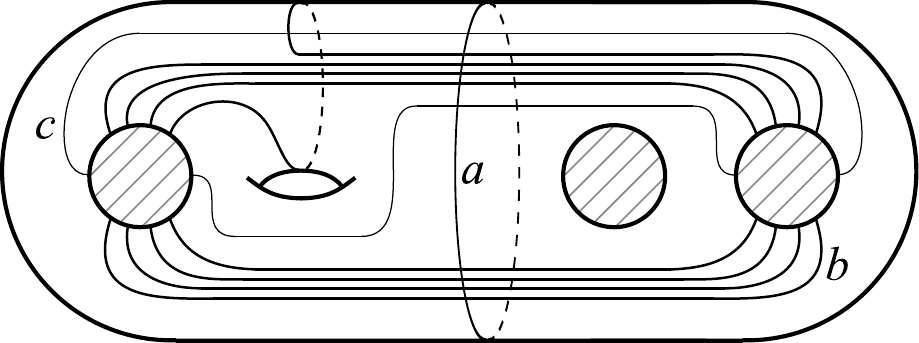}
\caption{Circles $a,b,c$ -- Example \ref{ex:ex2}.}\label{xr03} %
\end{center}
\end{figure}
In particular $I(a,b)=8$ and $I(c,a)=2>I(c,b)=1$. However, checking the following is straightforward:
\[I(t_a(c),a)=2>I(t_a(c),b)=1.\]
The above example can be generalized in the obvious way (by changing $b$) to the example where 
$I(a,b)=2n$, $I(c,a)=2>I(c,b)=1$ and $I(t_a(c),a)=2>I(t_a(c),b)=1$, where $n\geq 1$.
\end{ex}
The above examples are 
disappointing, because they show that the geometric intersection number is too weak 
to notice the action of a twist. Moreover, this situation can happen for arbitrary large complexity [that is for 
arbitrary large values of $I(a,b)$ and $I(c,a)$].
\begin{ex}\label{ex:ex3}
 Let $a,c$ be two-sided circles indicated in Figure \ref{xr02}. 
 \begin{figure}[h]
\begin{center}
\fig{0.44}{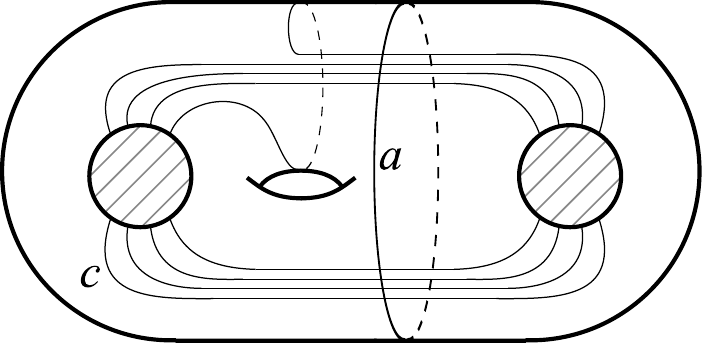}
\caption{Circles $a,c$ -- Example \ref{ex:ex3}.}\label{xr02} %
\end{center}
\end{figure}
The action of $t_a$ on $c$ is trivial because $a$ bounds a \Mob; this is the case even 
though $I(a,c)=8$ (or in general $I(a,c)=2n$, $n\geq 1$). Such a situation can not happen on an oriented surface 
$S$ -- if $I(a,c)>0$ for some curve $c$ on $S$, then $t_a$ is automatically nontrivial. 
\end{ex}
 \section{Joinable segments of $a$ and $b$}\label{sec:joinable}%
For the rest of the paper assume that $a$ and $b$ are two generic two-sided circles in 
a nonorientable surface $N$ such that $|a\cap b|=I(a,b)\geq 2$. 

Following \cite{Stukow_twist} by a \emph{segment} of $b$ (with respect to $a$) we mean any unoriented arc $p$ of $b$ that 
satisfies $a\cap p=\partial p$. Similarly, we define an \emph{oriented segment} of $b$. If $p$ is an oriented segment, then by $-p$, we denote the segment equal to $p$ as an oriented segment but with the opposite orientation. 

We call a segment $p$ of $b$ \emph{one-sided}
[\emph{two-sided}] if the union of $p$ and an arc of $a$ connecting $\partial p$ is a one-sided [two-sided] circle. An oriented 
segment is one-sided [two-sided] if the underlying unoriented segment is one-sided [two-sided].

If $P,Q\in a\cap b$ are two intersection points of $a\cap b$ consecutive on $b$, then by $PQ$ we denote an oriented segment of $b$ with endpoints $P$ and $Q$. Oriented segments $PP'$ and $QQ'$ of $b$ are called \emph{adjacent} if both are one-sided and 
an open disk $\Delta$ exists
on $N$ with the following properties
\begin{enumerate}
 \item $\partial \kre{\Delta}$ consists of the segments $PP'$, $QQ'$ of $b$ and the arcs $PQ$, $P'Q'$ of $a$;
 \item $\Delta$ is disjoint from $a\cup b$ (Figure \ref{r00}).
\end{enumerate}
\begin{figure}[h]
\begin{center}
\fig{0.85}{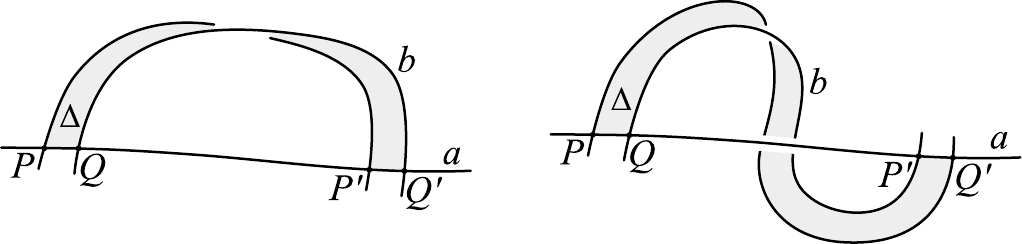}
\caption{Adjacent segments of $b$.}\label{r00} %
\end{center}
\end{figure}
Oriented segments $p\neq q$ are called \emph{joinable} if 
oriented segments $p_1,\ldots,p_k$ exist such that $p_1=p$, 
$p_k=q$ and $p_i$ is adjacent to $p_{i+1}$ for $i=1,\ldots,k-1$ (Figure \ref{r00a}).

Unoriented segments are called adjacent [joinable] if they are adjacent [joinable] as oriented segments for some choice of 
orientations.

In exactly the same way we define segments of $a$ (with respect to $b$) and their properties.

\begin{uw}\label{rem:adjac}
The main reason for the importance of adjacent/joinable segments of $b$ is that they provide natural 
reductions of the intersection points of $t_a(b)$ and $b$ (Figure \ref{r00a}). 
\begin{figure}[h]
\begin{center}
\fig{0.9}{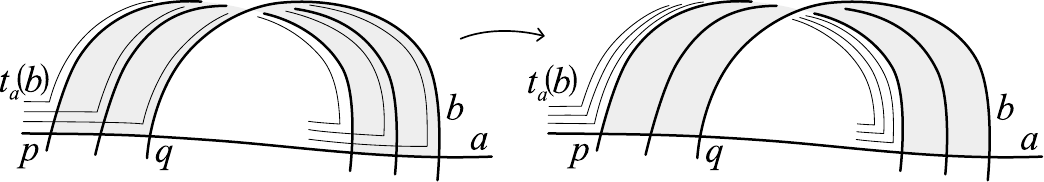}
\caption{Joinable segments of $b$ and reduction of intersection points of $b$ and $t_a(b)$.}\label{r00a} %
\end{center}
\end{figure}
In fact, as observed in \cite{Stukow_twist}, these
segments are the only nontrivial source of such reductions. 
\end{uw}
Let us recall some basic properties of joinable segments.
\begin{prop}[Lemmas 3.4, 3.7 and 3.8 of \cite{Stukow_twist}]\label {prop:prop:segments}\quad
 \begin{enumerate}
  \item Initial [terminal] points of oriented joinable segments of $b$ are on the same side of $a$.
  \item Let $p$ and $q$ be oriented segments such that $q$ begins at the terminal point of $p$ (this includes the case $q=-p$). Then $p$ and $q$ are not joinable.\qed
 \end{enumerate}
\end{prop}
Still following \cite{Stukow_twist}, by a \emph{double segment} of $b$, we mean an unordered pair of two 
different oriented segments of $b$ that 
have the same initial point. Exactly $I(a,b)$ double segments exist, which correspond to intersection 
points of $a$ and $b$. 

Two double segments are called \emph{joinable} if an oriented segment $p$ exists in the first 
double segment and $q$ exists in the other such that $p$ and $q$ are joinable. 

Studying the action of a twist $t_a$ on a circle $b$ is important to obtain some obstructions for possible 
reductions of intersection points between $t_a(b)$ and $b$. The basic result in this direction is the 
following proposition:
\begin{prop}[Lemma 3.9 of \cite{Stukow_twist}] \label{prop:sing:dbl:seg}
 Suppose $I(a,b)\geq 2$. Then, for each double segment $P$, a double segment $Q\neq P$ exists, which 
 is not joinable to $P$.\qed
\end{prop}
 \section{Curves in the neighborhood of $a\cup b$}\label{sec:curves}%
A regular neighborhood $N_{a\cup b}$ of $a\cup b$ is fixed. Topologically, $N_{a\cup b}$ is the union of regular neighborhoods $N_a$ and $N_b$ of $a$ and $b$
respectively. By changing $N_a,N_b$ and $N_{a\cup b}$ into their closures we can assume that all these sets are closed.
If we define 
\[N_{a\bez b}=\kre{N_a\bez N_b},\ N_{b\bez a}=\kre{N_b\bez N_a},\ N_{a\cap b}=N_a\cap N_b,\]
then 
\[N_{a\cup b}=N_a\cup N_b=N_{a\bez b}\cup N_{b\bez a}\cup N_{a\cap b},\]
where each three sets on the right hand side consist of $I(a,b)$ disks with disjoint interiors. These disks correspond to the intersection points 
of $a$ and $b$ (Figure~\ref{r03a}). 
\begin{figure}[h]
\begin{center}
\fig{0.45}{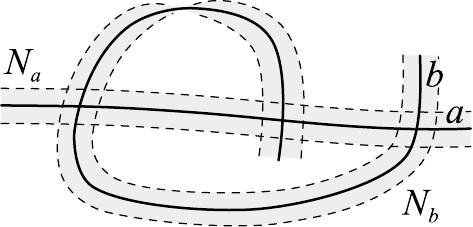}
\caption{Neighborhood of $a\cup b$ as the union of \mbox{$3\cdot I(a,b)$} rectangles.}\label{r03a} %
\end{center}
\end{figure}
We consider these disks as rectangles with two opposite sides parallel to $a$
and the other two parallel to $b$. 
The rectangles in $N_{a\bez b}$ and $N_{b\bez a}$ have a one-to-one correspondence with the segments of $a$ and $b$, respectively. 

If $r$ is one of the rectangles that constitute $N_{a\cup b}$ and $c$ is a circle in $N_{a\cup b}$ that intersects $a\cup b$ transversally, then by an \emph{arc of} $r\cap c$, we mean a connected component of $r\cap c$. 

Let ${\cal C}$ be the family of generic circles on $N$ that satisfy the following properties:
\begin{enumerate}
 \item Each circle in ${\cal C}$ is contained in $N_{a\cup b}$ and intersects $a\cup b$ transversally.
 \item Each intersection point of $c$ and $a\cup b$ is contained in $N_{a\cap b}$.
 \item If $c\in{\cal C}$ and $r$ is one of the rectangles in $N_{a\cup b}$,
 then each arc of 
 $c\cap r$ has endpoints on two different sides of $r$.
\end{enumerate}
The third condition simply means that $c$ does not turn back when crossing a rectangle. 
Each generic circle contained in $N_{a\cup b}$ is obviously isotopic to a circle in ${\cal C}$.

Let $c\in{\cal C}$ and let $r$ be one of the rectangles in $N_{a\cup b}$. If an arc of $c$ contained
in $r$ crosses both sides of $r$ parallel to $a$, then we say that $r$ contains an arc of $c$ 
\emph{parallel} to $b$.

If every rectangle in $N_{b}$ contains an arc of $c$ parallel to $b$, then we say that \emph{$c$ winds around $b$}.
Clearly, the sufficient condition for a circle $c\in{\cal C}$ to wind around $b$ is that each rectangle in 
$N_{a\cap b}$ contains an arc of $c$ parallel to $b$. 

If $r$ is a rectangle in $N_{b\bez a}$ and $r$ contains an arc $q$ of $c$ parallel to $b$,
then we say that $q$ is a \emph{segment} of $c$. Moreover, if $p$ is a segment of $b$ that corresponds to $r$, then we say that $q$ is parallel to $p$.
Similarly, we define segments of $c$ parallel to segments of $a$.
%
%
%
\begin{lem}\label{lem:ExtBigon1}
 Let $c\in {\cal C}$ such that $c$ winds around $b$, and let $a'$ be one of the components of $\partial N_a$. 
 If $|c\cap a|=I(c,a)$ and $\Delta$ is a bigon formed by $c$ and $a'$, then $\Delta\subset N_a$.
\end{lem}
\begin{proof}
 Suppose that a bigon $\Delta$ with sides $p\subset a'$ and $q\subset c$ exists, which is not contained 
 in $N_a$. If $\Delta$ and $N_a$ are on the same side of $p$ (Figure \ref{r11}), then we can 
 find a smaller bigon $\Delta'\subset \Delta$ with sides $p'\subset \partial N_a$ and $q'\subset c$ such that
 $\Delta'$ and $N_a$ are on different sides of $p'$ (because $\Delta\setminus N_a\neq \emptyset$). 
 \begin{figure}[h]
\begin{center}
\fig{0.9}{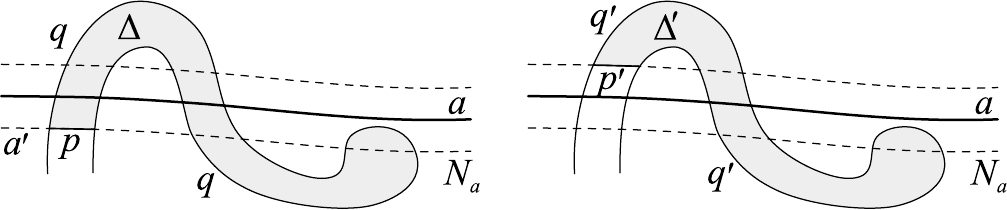}
\caption{Case of $\Delta$ and $N_a$ being on the same side of $a'$, Lemma \ref{lem:ExtBigon1}.}\label{r11} %
\end{center}
\end{figure}
 Hence we can assume that $\Delta$ and $N_a$ are on different sides of $p$.
 
 If $c$ intersects the interior of $p$, then we can pass to a smaller bigon that is still not contained in $N_a$. Hence, we can assume that $p\cap c$ consists of two points $A$ and $B$ (note that $q$ may still intersect $a'$).
 
 Let $r_A$ and $r_B$ be rectangles of $N_{a\cap b}$ that correspond to $A$ and $B$, respectively. If $r_A=r_B$, then 
 the segments of $q$ that start at $A$ and $B$ terminate at the same rectangle of $N_{a\cap b}$ (Figure \ref{r11a}).
 Hence, we can pass to 
 a smaller bigon $\Delta'\subset \Delta$ by removing these segments of $q$. The obtained bigon $\Delta'$ is still not 
 contained in $N_a$ because this would imply that $c$ is not in ${\cal C}$ ($c$ would need to turn back in one 
 of the rectangles of $N_a$). 
 \begin{figure}[h]
\begin{center}
\fig{0.99}{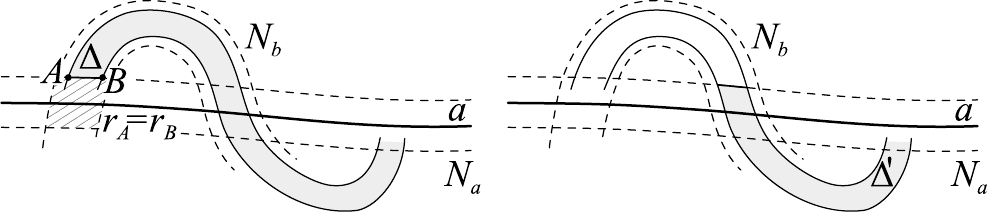}
\caption{The case of $r_A=r_B$, Lemma \ref{lem:ExtBigon1}.}\label{r11a} %
\end{center}
\end{figure}
 Hence we can assume that $r_A\neq r_B$. 
 
 Let $c_A,c_B\subset N_a$ be arcs of $c$ that start at $A$ and $B$, respectively.
 
 Recall that we assumed that $c$ winds around $b$. Hence, $r_A$ and $r_B$ contain arcs of $c$ that are 
 parallel to $b$. Therefore, $c_A$ either crosses $a$ in $r_A$ or $c_A$ turns in $r_A$ in the direction 
 of $p$, and after running parallel to $p$, $c$ must turn and cross $a$ in $r_B$. In fact,
 $c$ cannot turn towards $p$ and it cannot cross $r_B$ because $r_B$ contains an arc of $c$ parallel to $b$ (Figure \ref{r12}(i)).
 \begin{figure}[h]
\begin{center}
\fig{0.93}{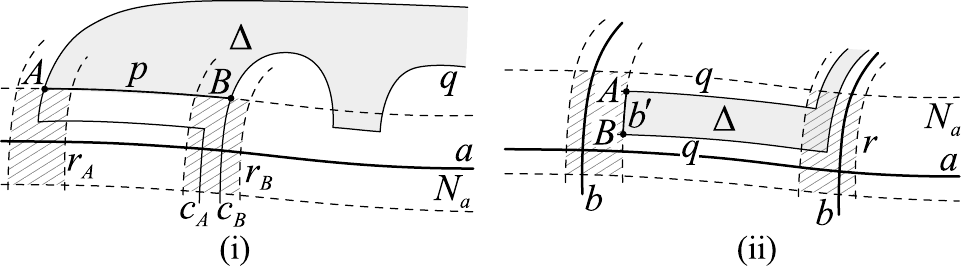}
\caption{Configurations of arcs, Lemmas \ref{lem:ExtBigon1} and \ref{lem:ExtBigon:b:side}.}\label{r12} %
\end{center}
\end{figure}
 Similar analysis applied to $c_B$ shows that the arc of $c\cap N_a$ that contains $c_B$ must intersect $a$; 
 it can do it in $r_B$, or in $r_A$ after running parallel to $p$. 
 However, this implies that $c$ and $a$ form  
 a bigon, which is a contradiction.
%
%
%
\end{proof}
\begin{lem}\label{lem:ExtBigon:b:side}
 Let $A$ and $B$ be the endpoints of an arc $b'$ contained in one of the components of $\partial N_b\cap N_a$. Let $c\in {\cal C}$ be such that $c$ winds around $a$, and let $q\subset c$ be an arc with endpoints $A$ and $B$ which starts and ends on the same side of the component of $\partial N_b$ containing $b'$. Then $b'$ and $q$ do not form a bigon with interior disjoint from $b\cup c$.
%
%
%
\end{lem}
\begin{proof}
Suppose to the contrary that $b'$ and $q$ form a bigon $\Delta$ with interior disjoint from $b\cup c$
and 
consider the arcs of $q$ that start at $A$ and $B$ [Figure \ref{r12}(ii)]. If these arcs enter
 some rectangle $r$ of $N_{a\cup b}$, then they must be parallel in $r$, that is, they are disjoint and 
 intersect the same sides of $r$. Clearly, this condition is true for rectangles in $N_{a\cup b}\bez N_{a\cap b}$ (since $c\in {\cal C}$), and for rectangles in $N_{a\cap b}$, this follows from our assumptions that the interior of $\Delta$ is disjoint from $b\cup c$ and that $c$ winds around $a$.
 However, this implies that the arcs of $q$ which start $A$ and $B$ will never meet.
\end{proof}
Let $c\in {\cal C}$ and $p$ be one of the arcs of $c\cap N_a$. Four different possible 
configurations of $p$ exist [Figure \ref{r13}] and are referred to as types~A--D.
\begin{figure}[h]
\begin{center}
\fig{0.96}{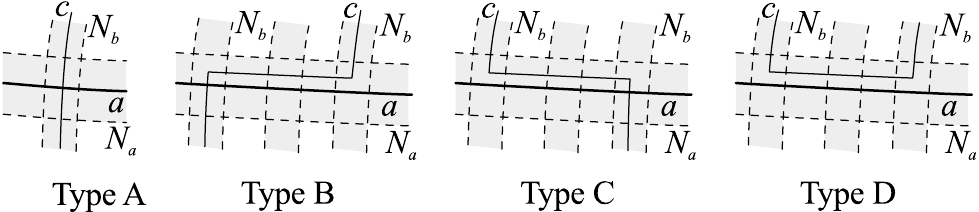}
\caption{Possible configurations of arcs of $c\cap N_a$.}\label{r13} %
\end{center}
\end{figure}
\begin{uw}\label{uw:sig_inter_rect}
 If $c$ winds around $b$, then arcs of $c\cap N_a$ of types B--D can pass through only one
rectangle in $N_{a\bez b}$. Otherwise, $c$ would intersect itself.
\end{uw}
%
\section{Rigidity of circles in ${\cal C}$}\label{sec:Rigid}
\begin{uw}
As we mentioned in Remark \ref{rem:adjac}, adjacency between segments of $b$ is the only nontrivial source of reductions of the intersection points between $b$ and $t_a(b)$. However, if we consider the intersection points between $b$ and $t_a(c)$ for $c\in{\cal C}$, then other kinds of reductions exist. For example, if $\Delta$ is a component of $N\bez N_{a\cup b}$, which is a disk, then $t_a(c)$ and $b$ may reduce along $\Delta$ (Figure \ref{r00b}).
\begin{figure}[h]
\begin{center}
\fig{0.94}{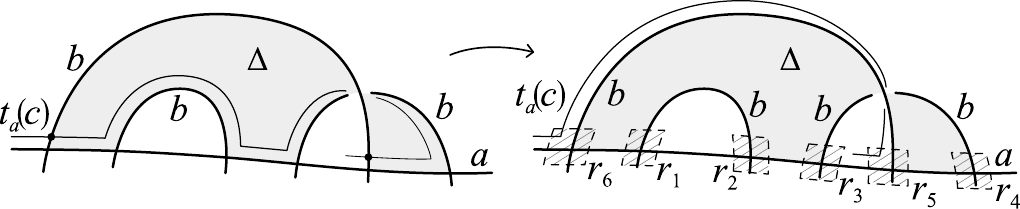}
\caption{Exterior hexagon $\Delta$ and the reduction of intersection points between $b$ and $t_a(c)$.}\label{r00b} %
\end{center}
\end{figure}
As we will see later, this type of reduction is rather exceptional, but we need additional definitions 
to control it. 
\end{uw}
Suppose that $\Delta$ is a component of $N\bez N_{a\cup b}$, which is a disk. If the boundary of $\kre{\Delta}$ intersects exactly $n$ boundaries of rectangles $r_1,r_2,\ldots,r_n$ in $N_{a\cap b}$, then we say that $\Delta$ is an \emph{exterior $n$-gon} with \emph{vertices} $r_1,r_2,\ldots,r_n$ (Figure~\ref{r00b}). Note that $r_1,\ldots,r_n$ does not need to be pairwise distinct ($\kre{\Delta}$ may intersect $r_i$ in each of its four corners).

Let $p$ be an arc of $c\in{\cal C}$, which is parallel to $b$ in a rectangle $r$ of $N_{a\cap b}$. Fix some orientation of $p$ and follow $p$ to the rectangle $r_1$ of $N_{a\cap b}$ following $r$. We say that $p$ is \emph{rigid} in $r$ with respect to $b$ if
$p$ is parallel to $b$ in $r_1$. 
Equivalently (from the perspective of $p$ intersecting $N_a$), $p$ is of type A in $r$ and $r_1$.


We say that $c\in{\cal C}$ \emph{winds strongly} around $b$ if for every rectangle $r$ in $N_{a\cap b}$
an oriented arc $p$ that is parallel to $b$ in $r$ exists, such that both $p$ and $-p$ are rigid in $r$.
Equivalently, for each three rectangles $r_1,r_2,r_3$ of $N_{a\cap b}$, which are consecutive along $b$, an arc of $c\cap N_{b}$ exists, which is parallel to $b$ in $r_1, r_2$, and $r_3$. 
In particular, if a circle $c\in{\cal C}$ winds strongly around $b$, then $c$ winds around $b$.


%
%

Let $R$ be a double segment of $b$ and let $p\neq q$ be oriented segments of $b$ that are not contained in $R$ and do not start at the same point of $a\cap b$. Assume also that $p$ and $q$ start on different sides of $a$. If we fix an orientation of $N_a$, then two possible mutual positions of $p,q$, and $R$ exist (Figure \ref{r11b}). 
\begin{figure}[h]
\begin{center}
\fig{0.7}{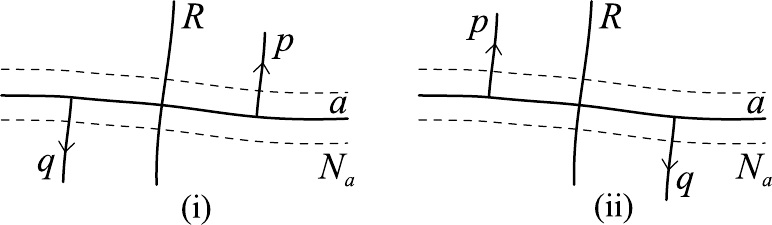}
\caption{Two possible orientations of $\{p,q,R\}$.}\label{r11b} %
\end{center}
\end{figure}
We say that the triple $\{p,q,R\}$ is \emph{positively oriented} if the configuration is as in Figure \ref{r11b}(i).

As we will see later, some special configurations of generic two-sided circles $a,b$ exist in $N$, which will require additional analysis (Sections \ref{sec:specA}--\ref{sec:2}). These special configurations are defined by the following properties:
\begin{itemize}
  \item[(S1)] $I(a,b)\geq 4$ and oriented segments $p,q$ of $b$ exist, which start on different sides of $a$ such that each
 double segment of $b$ contains an oriented segment joinable to $p$ or $q$ (see Figure \ref{r21} in Section \ref{sec:specA}).
 \item[(S2)] $I(a,b)\geq 4$ and oriented segments $p,q$ of $b$, and a double segment $R$ of $b$ exist, such that: $a$ and $b$ are not in the special position (S1), $p$ and $q$ start on different sides of $a$, $p$ starts and terminates on different sides of $a$, each double segment of $b$ different from $R$ contains an oriented segment joinable to $p$ or $q$, $\{p,q,R\}$ is positively oriented (see Figure \ref{r21} in Section \ref{sec:specA}).
 \item[(S3)] $I(a,b)\geq 4$ and there are oriented segments $p,q$ of $b$ and a double segment $R$ of $b$ such that: $p$ starts and terminates on one side of $a$, $q$ starts and terminates on the other side of $a$,
 each double segment of $b$ different from $R$ contains an oriented segment joinable to $p$ or $q$, $\{p,q,R\}$ is positively oriented (see Figure~\ref{r21b} in Section \ref{sec:specB}).
\end{itemize}
If one of the ordered pairs of circles $(a,b)$ or $(b,a)$ satisfies one of the above conditions, then we say that the unordered pair  $\{a,b\}$ is \emph{special}. 

Let $X_a$ be the set of isotopy classes of circles $c$ in $N$, which satisfy the following conditions:
\begin{enumerate}
 \item $c\in {\cal C}$,
 \item $I(c,a)=|c\cap a|$, $I(c,b)=|c\cap b|$,
 \item $I(c,a)<I(c,b)$,
 \item $c$ winds strongly around $a$.
\end{enumerate}
Similarly, we define $X_b$ by requiring (1)--(2) above and additionally
\begin{enumerate}
 \item[(3')] $I(c,b)<I(c,a)$,
 \item[(4')] $c$ winds strongly around $b$.
\end{enumerate}
 \section{The case of $I(a,b)\geq 4$ and $\{a,b\}$ not special} \label{sec:3}
The main goal of this section is to prove the following:
\begin{prop}\label{Prop:main:g3}
 Let $a$ and $b$ be two generic two-sided circles in $N$ such that $I(a,b)\geq 4$ and $\{a,b\}$ is not special. Then for any integer $k\neq 0$ we have
 \[t_a^k(X_b)\podz X_a\quad\text{and}\quad t_b^k(X_a)\podz X_b.\]
\end{prop}
 \begin{proof}
 Of course 
 proving that $t_a^k(X_b)\podz X_a$ is sufficient. No canonical choice exists for the 
 orientation of the neighborhood $N_a$. However, in our figures, we will assume that $t_a^k$ twists to the right.

 
 \emph{Construction of $t_a^k(c)$.}
 Fix a circle $c\in X_b$. It is enough to prove that $t_a^k(c)\in X_a$.
 
 We begin by constructing the circle $d=t_a^k(c)$. Outside $N_a$ and on each arc of $d\cap N_a$ of type D, $d$ 
 is equal to $c$. For each arc of $c\cap N_a$ 
 of types A--C, $d$ circles $|k|$ times around $a$. In particular $d$ winds around $a$ and 
 \[\begin{aligned}
   I(d,a)&=|d\cap a|=|c\cap a|=I(c,a)\\
   |d\cap b|&=I(c,a)\cdot I(a,b)\cdot |k|.
   \end{aligned}
\]
  Now the problem is that $d$ may not be an element of ${\cal C}$ and $d$ does not need to be in a minimal 
  position with respect to $b$. 
 
 Before we start to reduce $d$, observe that if an arc of $d$ enters $N_a$ and turns to the left, then
 after passing through one rectangle in $N_{a\bez b}$ it must 
 turn back or leave $N_a$ through the same side of $N_a$ as it entered $N_a$ (Figure \ref{r08}). 
 \begin{figure}[h]
 \begin{center}
\fig{0.6}{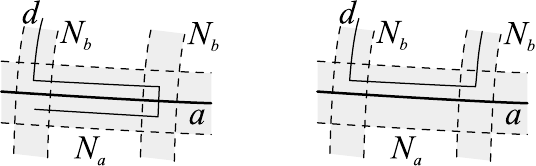}
 \caption{Arcs of $d$ turning to the left in $N_a$.}\label{r08} %
 \end{center}
 \end{figure}
 In fact, arcs 
 of $d$ turning to the left in $N_a$ came from arcs of $c\cap N_a$ of types C and D. As we observed in Remark \ref{uw:sig_inter_rect}  
 such arcs can pass through only one rectangle in $N_{a\bez b}$.

 \emph{Reduction of type I.}
 Suppose that one of the rectangles $r$ in $N_{a\cup b}$ contains an arc $p$ of $d$ such that the endpoints of $p$ are on the same
 side of $r$ ($d$ turns back in $r$). Clearly, this situation cannot happen for $r$ being one of the rectangles in $N_{b\bez a}$ (because in such a rectangle, 
 $d$ coincides with $c$) or $r$ being a rectangle in $N_{a\bez b}$ (by construction $d$ runs parallel to $a$ in 
 each such rectangle). Hence, $r$ must be a rectangle in $N_{a\cap b}$ and $p$ must intersect the $b$-side of $r$ (otherwise
 $c$ would not be an element of ${\cal C}$). Hence, we have the situation illustrated in Figure \ref{r04}, and
  we can replace $d$ with the circle $d'$ shown in the same figure. 
 \begin{figure}[h]
 \begin{center}
\fig{0.6}{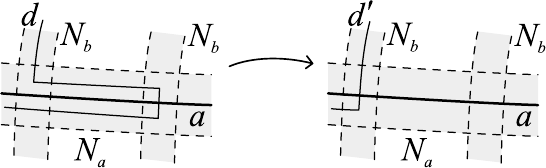}
 \caption{Reduction of type I.}\label{r04} %
 \end{center}
 \end{figure}
 In such a case, we say that we reduced $d$ by a \emph{reduction of type I}. 
 
 Reductions of type I correspond to arcs of $c\cap N_a$ of type C. Hence, on each arc of 
 $d\cap N_a$, at most one reduction of type I exists.

 
 Let $d_1$ be the circle obtained from $d$ by performing all possible reductions of type I. In particular, 
 $d_1\in{\cal C}$. 
 \begin{uw}\label{uw:arc:turn:left:D}
  The only arcs of $d_1\cap N_a$ that turn to the left after 
 entering $N_a$ are arcs that correspond to (in fact, equal to) arcs of $c\cap N_a$ of type D.
 \end{uw}
 
 We now argue that $d_1$ winds around $a$. In fact, if we fix a rectangle $r$ in $N_{a\cap b}$ 
 and $r'$ is another rectangle in $N_{a\cap b}$, then (because $c$ winds around $b$) $r'$ contains an arc $q$ of $c$
 parallel to $b$ (Figure \ref{r07}). 
 \begin{figure}[h]
 \begin{center}
\fig{0.65}{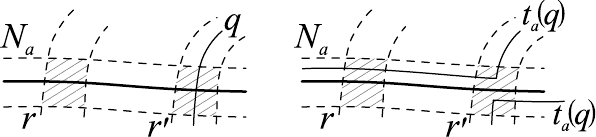}
 \caption{Action of $t_a$ on an arc $q$ in a rectangle $r'$.}\label{r07} %
 \end{center}
 \end{figure}
 Now $t_a^k(q)$ is strictly monotone in $N_a$ with respect to $a$. Hence, this arc does not admit a reduction of type I. In particular, $t_a^k(q)$ gives an arc of $d_1$, which is parallel to $a$ in $r$.

 For further reference, note the following observation:
 \begin{lem}\label{lem:d:eq:c}
 Let 
 $q\subset d_1$ be an arc with endpoints $A,B\in \partial N_a$ such that $q\cap b=\emptyset$. Assume also that $q$ intersects both $N_{a\bez b}$ and $N_{b\bez a}$. 
Then $q$ is an arc of $c$. 
 \end{lem}
\begin{proof}
By construction, $d_1$ coincides with $c$ in every rectangle of $N_{b\bez a}$ and if $q$ enters $N_a$ in a rectangle $r$ (Figure \ref{r15c}), then 
$q$ must leave $r$ through the $b$-side of $r$ given that $d_1$ winds around $a$. 
\begin{figure}[h]
 \begin{center}
\fig{0.8}{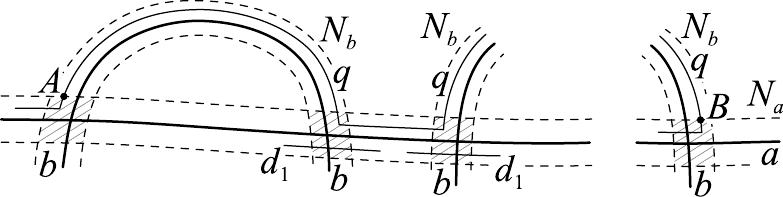}
 \caption{Configuration of arcs -- Lemma \ref{lem:d:eq:c}.}\label{r15c} %
 \end{center}
 \end{figure}
Afterwards, $q$ goes through one rectangle of $N_{a\bez b}$, and then it must turn and leave $N_a$ (because $q$ cannot intersect $b$). Hence, each arc of $q\cap N_a$ is an arc of type D, which proves that $q$ is an arc of $c$. 
\end{proof}
 
 \emph{Reduction of type II.}
Suppose that there exist arcs $p$ and $q$ of $b$ and $d_1$, respectively, such that 
\begin{itemize}
 \item $p$ and $q$ form a bigon with interior disjoint from $b\cup d_1$,
 \item $p\bez N_{a\cap b}$ is a subarc of a two-sided segment of $b$.
\end{itemize}
In such a case, we can remove the bigon formed by $p$ and $q$ (Figure~\ref{r05}), and we say that we reduced $d_1$ to 
$d_1'$ by a \emph{reduction of type II}.
 \begin{figure}[h]
 \begin{center}
\fig{0.97}{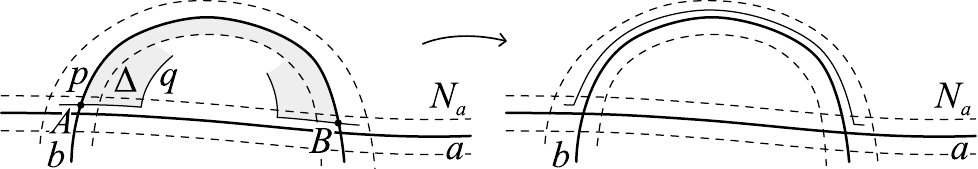}
 \caption{Reduction of II.}\label{r05} %
 \end{center}
 \end{figure}
 
Let us describe the possible reductions of type II in more detail. Let $A$ and $B$ be vertices of the bigon $\Delta$ 
formed by $p$ and $q$, and let $r_A,r_B$ be the rectangles of $N_{a\cap b}$ that contains $A$ and $B$, respectively. 
By Lemma \ref{lem:ExtBigon:b:side}, the arc $q_A\subset q$, which connects $A$ with the boundary of $N_a$, is either entirely contained in $r_A$ or it passes through one rectangle of $N_{a\bez b}$ and then leaves $N_a$ (it can pass at most one rectangle of $N_{a\bez b}$ because otherwise it would intersect $b$). In other words the situation illustrated in Figure \ref{r05a} is not possible. 
\begin{figure}[h]
 \begin{center}
\fig{0.27}{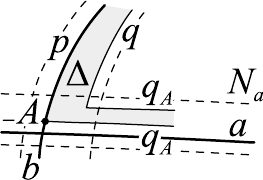}
 \caption{Types II and III reductions -- impossible configuration of arcs.}\label{r05a} %
 \end{center}
 \end{figure}
The same is true for the arc $q_B\subset q$, which connects $B$ with the boundary of $N_a$. For the same reason, either both arcs $q_A,q_B$ are entirely contained in $r_A,r_B$, or each of them passes through one rectangle of $N_{a\bez b}$.

Given that we assume that $p$ corresponds to a two-sided segment of $b$, at one endpoint of $q$, say $B$, $q$ turns to the left as it enters $N_a$.  We claim that $q$ consists of $q_A,q_B$ and a single segment of $c$ parallel to $b$. In order to prove this statement, showing that $q_B$ is entirely contained in $r_B$ is sufficient.  Suppose to the contrary that $q_B$ passes through a rectangle of $N_{a\bez b}$
(Figure \ref{r10}). 
\begin{figure}[h]
 \begin{center}
\fig{0.99}{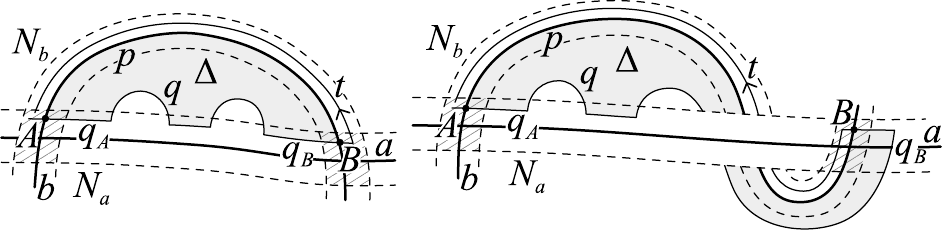}
 \caption{Reduction of type II -- impossible configurations of arcs.}\label{r10} %
 \end{center}
 \end{figure}
Given that $q_B$ turns to the left after entering $N_a$, by Remark \ref{uw:arc:turn:left:D}, this arc must be an arc of type D. Hence, after crossing $p$ in $B$ it
must follow an arc $t$ of $d_1$, which turns left in $r_B$ and leaves $N_a$ (it must leave $N_a$ in $r_B$ 
by Remark \ref{uw:sig_inter_rect}).
After leaving $N_a$, $t$ 
runs parallel to $p$ in a rectangle of $N_{b\bez a}$. In particular, $q_B$ and $t$ are arcs of $c$ [given that $(q_B\cup t)\cap N_a$ is an arc of type D]. Moreover, by Lemma \ref{lem:d:eq:c}, $q\bez (q_A \cup q_B)$ is also an arc of $c$. Hence, $(q\bez q_A)\cup t\subset c$ and an arc of $\partial N_a$ form a bigon $\Delta'$ not contained in $N_a$ (Figure \ref{r10a}). This finding contradicts Lemma \ref{lem:ExtBigon1}.
\begin{figure}[h]
 \begin{center}
\fig{0.99}{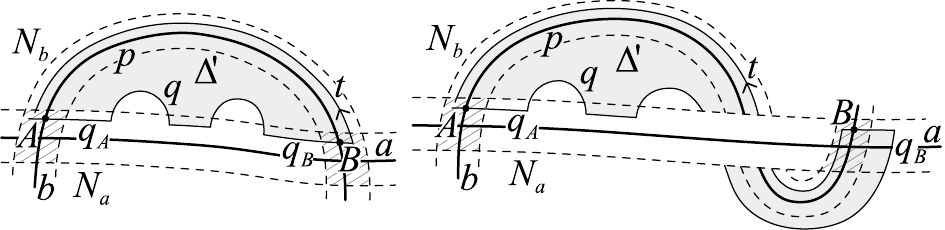}
 \caption{Reduction of type II -- extending $\Delta$ to $\Delta'$.}\label{r10a} %
 \end{center}
 \end{figure}
%
%

Therefore, we proved that if $d_1'$ and $d_1$ differ by a reduction of type II that corresponds to the bigon formed by arcs 
$p$ and $q$ of $b$ and $d_1$, respectively, then $p$ and $q$ are in the same rectangle of $N_{b\bez a}$, and $q$ is parallel to $p$.
This condition means that $d_1$ and $d_1'$ intersect rectangles in $N_{a\cup b}$ in exactly the same way. Hence, $d_1'\in {\cal C}$ and
$d_1'$ winds around $a$.

Let $d_2$ be the circle obtained from $d_1$ by performing all possible reductions of type II. As we observed, $d_2\in{\cal C}$
and $d_2$ winds around $a$.
\begin{uw}\label{rem:override}
In the following, we use Lemma \ref{lem:d:eq:c} with $d_2$ instead of $d_1$. This approach is somewhat problematic, because some arcs of $d_1$ that satisfy the assumptions of that lemma may admit reductions of type II. To solve this problem, we mimic reductions of type II on the level of $c$. To be more precise, if $q$ is an arc of $c\cap N_{b\bez a}$ which as an arc of $d_1$ admits a reduction of type II, then we push $q$ so that it coincides with the corresponding arc of $d_2$ (Figure \ref{r05b}). 
\begin{figure}[h]
 \begin{center}
\fig{0.97}{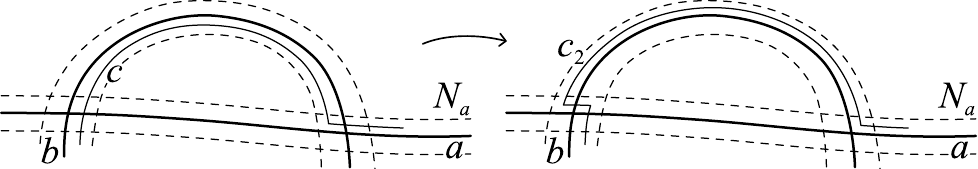}
 \caption{Applying reductions of type II to $c$.}\label{r05b} %
 \end{center}
 \end{figure}
Then we extend this push to the isotopy of $c$. If $c_2$ is a circle obtained from $c$ by all possible reductions of type II (in the sense described above), then $c_2$ still winds around $b$. Hence, it satisfies the assumptions of Lemma \ref{lem:ExtBigon1}. Moreover, we have the following replacement for Lemma \ref{lem:d:eq:c}:
\end{uw}
\begin{lem}\label{lem:d:eq:c2}
Let 
 $q\subset d_2$ be an arc with endpoints $A,B\in \partial N_a$ such that $q\cap b=\emptyset$. Assume also that $q$ intersects both $N_{a\bez b}$ and $N_{b\bez a}$. 
Then $q$ is an arc of $c_2$. \qed 
 \end{lem}
\emph{Reduction of type III.} 
Suppose that there exist arcs $p$ and $q$ of $b$ and $d_2$ respectively such that 
\begin{itemize}
 \item $p$ and $q$ form a bigon with interior disjoint from $b\cup d_2$,
 \item $p\bez N_{a\cap b}$ is a subarc of a one-sided segment of $b$.
\end{itemize}
In such a case we can remove the bigon formed by $p$ and $q$, see Figure~\ref{r06}. 
 \begin{figure}[h]
 \begin{center}
\fig{0.97}{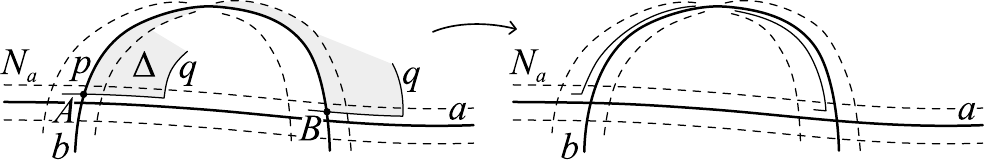}
 \caption{Reduction of type III.}\label{r06} %
 \end{center}
 \end{figure}
We say that we reduced $d_2$ to $d_2'$ by a \emph{reduction of type III}.

Let us attempt to understand reductions of type III in further detail. 

Let $A$ and $B$ be the vertices of the bigon $\Delta$ 
formed by $p$ and $q$, and let $r_A,r_B$ be the rectangles of $N_{a\cap b}$ that contains $A$ and $B$, respectively. By Lemma \ref{lem:ExtBigon:b:side}, the arc $q_A\subset q$ that connects $A$ with the boundary of $N_a$ is either entirely contained in $r_A$ or it passes through one rectangle of $N_{a\bez b}$ and then leave $N_a$ (it can pass through at most one rectangle of $N_{a\bez b}$ because otherwise it would intersect $b$). In other words, the situation illustrated in Figure \ref{r05a} is not possible. The same is true for the arc $q_B\subset q$, which connects $B$ with the boundary of $N_a$. For the same reason, either both arcs $q_A,q_B$ are entirely contained in $r_A,r_B$, or each arc passes through one rectangle of $N_{a\bez b}$.

If $q\bez (q_A\cup q_B)$ intersects $N_a$, then by Lemma \ref{lem:d:eq:c2}, this is an arc of $c_2$ (see Remark \ref{rem:override}). In particular, all components of $(q\bez (q_A\cup q_B))\cap N_a$ are arcs of type D (on arcs of other types, $d_2$ does not coincide with $c_2$).


Observe also that $q$ turns to the right when it enters $N_a$ (because $q$ is one-sided, it enters $N_a$ in the same way on both ends). In fact, $q_A$ and $q_B$ would be arcs of type D otherwise (Remark \ref{uw:arc:turn:left:D}), which are not involved in reductions of type II. Hence, $q_A$ and $q_B$ would be arcs of $c$ (Figure \ref{r06a}).
\begin{figure}[h]
 \begin{center}
\fig{0.48}{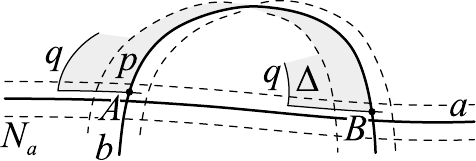}
 \caption{Reduction of type III -- impossible configuration of arcs.}\label{r06a} %
 \end{center}
 \end{figure}
Therefore, 
the whole $q$ would be an arc of $c$, and that would imply that $c$ and $b$ form a bigon.
\begin{uw}\label{uw:Red3:IsInC}
Observe that $d_2'\in{\cal C}$. In fact, if 
$d_2'$ turns back in one of the rectangles $r$ of $N_{a\cup b}$, then $r$ must be a rectangle that contains 
one of the vertices of the bigon, which leads to the reduction (in all other rectangles, we either did not 
change anything, or $d_2'$ runs parallel to $b$ in them). Since $d_2'$ enters $r$ through the $a$-side, after turning back, it must leave $r$ also through the $a$-side. 
\begin{figure}[h]
 \begin{center}
\fig{0.95}{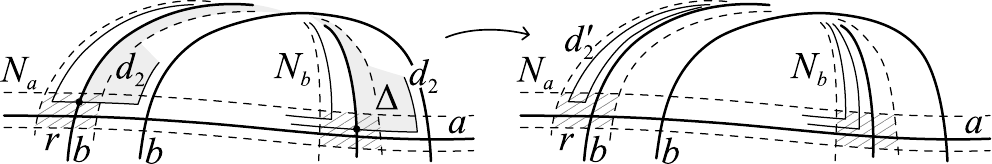}
 \caption{Reduction of type III -- impossible configuration of arcs.}\label{r15a} %
 \end{center}
 \end{figure}
 Hence, we have the situation shown in the right-hand side of Figure \ref{r15a}. The left-hand side of the same figure shows how the situation looked before the reduction.
 
 Let $t$ be an arc of $d_2$ following $q$ past the intersection point $p\cap q\cap r$. Arc $t$ turns right in $r$ and leaves $r$ as an arc parallel to $p$: see Figure \ref{r15b}. 
 \begin{figure}[h]
 \begin{center}
\fig{0.45}{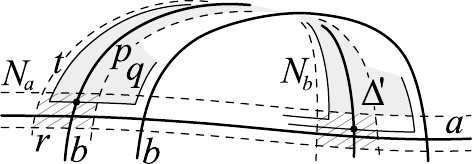}
 \caption{Reduction of type III -- impossible configuration of arcs.}\label{r15b} %
 \end{center}
 \end{figure}
 The same figure shows how the reduction disk $\Delta$ can be deformed to the closed disk $\Delta'$ with boundary composed of $q,t$, and an arc of $\partial N_a$. Note that $q\cap \Delta'$ is an arc of $c_2$ by Lemma \ref{lem:d:eq:c2} and $t$ is an arc of $c_2$ because $d_2$ intersects $r$ as an arc of type D. Therefore the existence of $\Delta'$ contradicts Lemma \ref{lem:ExtBigon1}.
Hence, we proved that $d_2'\in {\cal C}$ (in particular, $d_2'$ does not admit
reductions of type I).
\end{uw}
\begin{uw}\label{uw:Red3:NoNewRed2}
Reduction of type III does not create any new arcs of $d_2'\cap N_a$, which turn to the left 
after entering 
$N_a$. This reduction does not affect the segments of $d_2'$, which run parallel to two-sided segments of $b$. 
Hence, $d_2'$ does not admit reductions of type II.
\end{uw}
Before we go further, we divide possible reductions of type III into 3 subtypes. Let $p$ and $q$ be arcs of $b$ and $d_2$ respectively which define a reduction of type III.
\begin{itemize}
 \item If $p$ and $q$ enter $N_a$ in the same rectangles of $N_{a\cap b}$, then we say that $p$ and $q$ define a \emph{reduction of type IIIa}.
 \item If $p$ and $q$ enter $N_a$ in different rectangles of $N_{a\cap b}$ and $q$ meets a single rectangle of $N_{b\bez a}$, then we say that $p$ and $q$ define a \emph{reduction of type IIIc}.
 \item Otherwise we say that $p$ and $q$ define a \emph{reduction of type IIIb}.
\end{itemize}
 
 \emph{Reduction of type IIIa.}
 Let $d_3$ be the circle obtained from $d_2$ by performing all possible reductions of type IIIa. As we already 
observed, $d_3\in{\cal C}$ and $d_3$ does not admit reductions of types I, II and IIIa (Remarks \ref{uw:Red3:IsInC} and \ref{uw:Red3:NoNewRed2}). 
Moreover, $d_3$ intersects the rectangles of $N_{a\cup b}$ in exactly the same way as $d_2$, hence $d_3$ winds around $a$.
\begin{uw}
We now follow Remark \ref{rem:override} to obtain a version of Lemma~\ref{lem:d:eq:c} for $d_3$. As in the construction of $c_2$, we
mimic the reductions of type IIIa on $c_2$ (Figure \ref{r06c}). However, we do not need to mimic all reductions of type IIIa. 
\begin{figure}[h]
 \begin{center}
\fig{0.95}{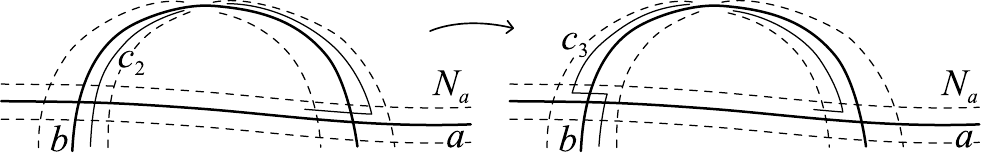}
 \caption{Applying reductions of type IIIa to $c_2$.}\label{r06c} %
 \end{center}
 \end{figure}
To be more precise, we apply to $c_2$ only these reductions of type IIIa, which are determined by an arc $q$ of $c_2$, which enters $N_a$ on one side as an arc of type D ($q$ cannot enter $N_a$ on both sides as an arc of type D: this condition would imply that $c$ and $b$ bound a bigon). If $q$ is an arc of $c$ that  enters $N_a$ on both sides as an arc of types A--C, then the corresponding arc of $d_3$ does not satisfy the assumptions of Lemma \ref{lem:d:eq:c} (on both ends, it intersects $b$). Hence, we do not consider these arcs as they do not mimic the setup of Lemma \ref{lem:d:eq:c}.

Let $c_3$ be the arc obtained from $c_2$ by reductions of type IIIa described above. Since $c_3$ still winds around $b$, it satisfies the assumptions of Lemma \ref{lem:ExtBigon1}.
As a replacement for Lemma \ref{lem:d:eq:c}, we have the following lemma:
\end{uw}
\begin{lem}\label{lem:d:eq:c3}
Let 
 $q\subset d_3$ be an arc with endpoints $A,B\in \partial N_a$ such that $q\cap b=\emptyset$. Assume also that $q$ intersects both $N_{a\bez b}$ and $N_{b\bez a}$. 
Then, $q$ is an arc of $c_3$. \qed 
 \end{lem}
\emph{Reduction of type IIIb.}
Let $d_3'$ be the circle obtained from $d_3$ by a single reduction of type IIIb.
As in the general definition of the reduction of type III, by $A$ and $B$ we denote vertices of the bigon formed by $p$ and $q$, $r_A$ and $r_B$ are rectangles of $N_{a\cap b}$ that contains $A$ and $B$ respectively, and $q_A,q_B$ are arcs of $q$ that connect $A$ and $B$ with the boundary of $N_a$.

\begin{uw}\label{uw:type3b:typeD:arc}
 By definition, $q\bez(q_A\cup q_B)$ is not a single segment of $d_3$ parallel to $b$. Hence, it intersects $N_a$ at least once. By Lemma \ref{lem:d:eq:c3}, $q\bez(q_A\cup q_B)$ is an arc of $c_3$, so it intersects $N_a$ only in arcs of type D (on arcs of other types, $d_3$ does not coincide with $c_3$).
\end{uw}

The existence of a bigon with vertices $A,B$ between $d_3$ and $b$ implies that $q$ and $p$ bound an exterior $n$-gon $\Delta$ (see Section \ref{sec:Rigid} and the right-hand side of Figure \ref{r00d}). 
\begin{figure}[h]
\begin{center}
\fig{0.94}{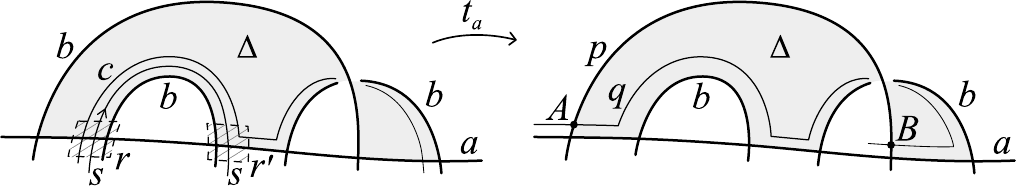}
\caption{Exterior hexagon $\Delta$ and the reduction of intersection points between $b$ and $t_a(c)$.}\label{r00d} %
\end{center}
\end{figure}
Moreover, if the reduction is not of type IIIc, then $n\geq 6$. In such a case, an oriented arc $s$ of $c$ which is of type A in a rectangle $r$ of $N_{a\cap b}$ (see the left-hand side of Figure \ref{r00d}) and  is rigid in $r$ with respect to $b$, cannot yield an arc of $d_3$, which allows a reduction of type IIIb (because if we follow $s$ to the next rectangle $r'$ of $N_{a\cap b}$ then $s$ is of type A in $r'$ and not of type D). 
\begin{uw}\label{uw:Rig:r3b:obst1}
 Our assumption that $c$ winds strongly around $b$ (see Section 6) implies that for each rectangle $r$ of $N_{a\cap b}$, an arc $s$ of $c$ exists, which is of type A in $r$ and such that $s$ yields an arc of $d_3$, which does not allow reductions of type IIIb on either side of $r$: for such $s$, an arc of $c$ that is rigid on both sides of $r$ should be selected.
\end{uw}
\begin{lem}\label{uw:Rig:r3b:obst2}
 If $s'$ is an arc of $c$ of type A in a rectangle $r$ of $N_{a\cap b}$, then $s'$ can yield a reduction of type IIIb only on one side of $r$. 
\end{lem}
\begin{proof}
 Suppose to the contrary that $s'$ yields an arc $s$ of $d_3$ which allows reductions of types IIIb on both sides of $r$ and assume that $r_1,r,r_2$ are rectangles of $N_{a\cap b}$ that are consecutive along $b$ (Figure \ref{r25}). 
 \begin{figure}[h]
\begin{center}
\fig{0.99}{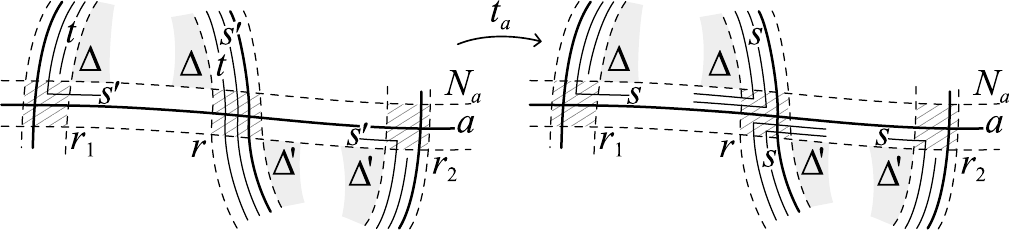}
\caption{Configuration of arcs -- Lemma \ref{uw:Rig:r3b:obst2}.}\label{r25} %
\end{center}
\end{figure}
 By Remark \ref{uw:type3b:typeD:arc}, $s$ coincides with $s'$ in $r_1$ and $r_2$ and enters each of these rectangles as an arc of type D. 
 However, this condition implies that each arc $t$ of $c$ that is of type A in $r$ must enter either $r_1$ or $r_2$ as an arc of type D (because it must follow $s'$ along the boundary of one of the exterior $n$-gons leading to reductions of $s$). This contradicts our assumption that $c$ winds strongly around $b$.
\end{proof}
\begin{lem}\label{lem:ends:ab}
 The arcs of $c\cap N_a$ that correspond to $q_A$ and $q_B$ are arcs of types A or B.
\end{lem}
\begin{proof}
If for example $q_A$ came from an arc of $c$ of type D, then by Lemma \ref{lem:d:eq:c3}, $q\bez q_B$ is an arc of $c_3$ and if we follow this arc past the point $A$, we obtain an arc of $c_3$
which, together with an arc of $\partial N_a$, bounds a disk $\Delta$ that is not contained in $N_a$ [Figure \ref{r06b}(i)]: this contradicts Lemma~\ref{lem:ExtBigon1}. 
\begin{figure}[h]
 \begin{center}
\fig{0.95}{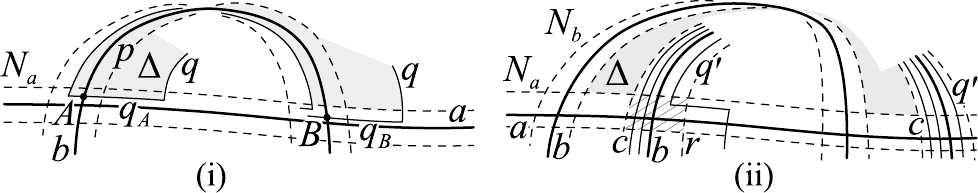}
 \caption{Reduction of type IIIb -- impossible configurations of arcs.}\label{r06b} %
 \end{center}
 \end{figure}

 As for arcs of type C, if for example $q\bez q_B$ was formed from the arc $q'$ of $c$ of type C and $r$ is the rectangle of $N_{a\cap b}$ in which $q_A$ meets $\partial N_a$, then all the arcs of $c$ that are parallel to $b$ in $r$ are on one side of $q'$ in $r$ and they are bounded along the boundary of $\Delta$ by $q'$ [Figure \ref{r06b}(ii)]. Hence, all these arcs must lead to arcs of $d_3$, which allow a reduction of type IIIb across $\Delta$. In particular, if we follow these arcs along the boundary of $\Delta$, then they enter $N_a$ as arcs of type $D$. However, this contradicts our assumption that 
  arcs that are rigid with respect to $b$ exist on both sides of $r$ (because $c$ winds strongly along $b$).

Hence we proved that $q_A$ and $q_B$ were formed from arcs of $c$ which are of types A or B in $r_A$ and $r_B$. 
\end{proof}
Note that the distinction between types B and C of arcs of $c\cap N_a$ is a consequence of our assumption that $t_a$ twists to the right in $N_a$.
\begin{uw}\label{uw:red3b:cont1}
 As a consequence of Lemma \ref{lem:ends:ab}, if $q'$ is an arc of $d_3'$ obtained from $q$ by a reduction of type IIIb, then the arcs of $d_3'\cap N_a$ that follow $q'$ (on each end) are not arcs of type D (because they intersect $a$). Hence, $q'$ is not involved in any further reductions of $d_3'$ of type IIIb (Remark \ref{uw:type3b:typeD:arc}). In other words, no cascade reductions of type IIIb exist: each arc of $d_3\cap N_{b\bez a}$ is involved in at most one such reduction. However, $q'$ may be further reduced by reductions of type IIIc. We postpone this problem to the analysis of reductions of type IIIc.
\end{uw}
\begin{uw}\label{uw:red3b:cont2}
If $q'$ is as in the previous remark (that is, $q'$ is an arc of $d_3'$ obtained from $q$ by a reduction of type IIIb), then $q'$ does not admit a reduction of type IIIa. This conclusion follows because the arcs of $d_3'\cap N_a$ that follow $q'$ (on each end) do not intersect $b$ in rectangles of $N_{a\cap b}$ in which they enter $N_a$ (Figure \ref{r06}). The only arc of $d_3'\cap N_{b\bez a}$ in which $d_3'$ differs from $d_3$ is $q'$. Thus, no new reductions of type IIIa exist on other arcs of $d_3'\cap N_{b\bez a}$.
\end{uw}
Let $d_4$ be the circle obtained from $d_3$ by performing all possible reductions of type IIIb. As we observed in the general analysis of reductions of type III,  $d_4\in{\cal C}$ (Remark \ref{uw:Red3:IsInC}) and $d_4$ does not admit any reductions of types I and II (Remarks \ref{uw:Red3:IsInC} and \ref{uw:Red3:NoNewRed2}). We also proved that $d_4$ does not admit reductions of type IIIa (Remark \ref{uw:red3b:cont2}).
\begin{uw}\label{rem:rec:a}
A segment $q$ of $d_4$ obtained by a reduction of type IIIb can not enter $N_a$ as an arc of type D (Lemma \ref{lem:ends:ab}). We will show later (Lemma~\ref{lem:one:new:inters}) that in such a case arcs that follow/precede $q$ in $N_a$ must intersect $b$. Hence Lemma \ref{lem:d:eq:c3} remains valid with $d_3$ replaced by $d_4$ (arcs modified by a reduction of type IIIb can not satisfy assumptions of that lemma).
\end{uw}
\emph{Reduction of type IIIc.}
Let $d_4'$ be the circle obtained from $d_4$ by a single reduction of type IIIc.
As in the general definition of the reduction of type III, by $A$ and $B$ we denote vertices of the bigon formed by $p$ and $q$, $r_A$ and $r_B$ are rectangles of $N_{a\cap b}$ that contain $A$ and $B$ respectively, and $q_A,q_B$ are arcs of $q$ that connect $A$ and $B$ with the boundary of $N_a$.
\begin{uw}\label{prop3b3c}
Reduction of type IIIc may be considered a special case of a reduction of type IIIb, where the exterior disk is a rectangle. For this reason, these two types of reductions have some common properties:
\begin{itemize}
 \item the arcs of $c$ that correspond to $q_A$ and $q_B$ cannot be of type~D (Lemma \ref{lem:ends:ab});
 \item if an arc $q'$ is an arc of $d_4'$ obtained by a reduction of type IIIc, then $q'$ is not involved in reductions of type IIIa nor IIIb (Remarks \ref{uw:red3b:cont1} and \ref{uw:red3b:cont2});
 \item $d_4'$ does not admit reductions of type IIIa nor IIIb.
\end{itemize}
The proofs of the above properties can be copied verbatim from the analysis of the reduction of type IIIb. Hence, we skip the proofs.
\end{uw}
Let $d_5$ be the circle obtained from $d_4$ by performing all possible reductions of type IIIc. As we observed in the general analysis of reductions of type III, $d_5\in{\cal C}$ (Remark \ref{uw:Red3:IsInC}) and $d_5$ does not admit any reductions of types I and II (Remark \ref{uw:Red3:NoNewRed2}). As we noted above (Remark \ref{prop3b3c}), $d_5$ also does not admit reductions of type IIIa and IIIb. 
\begin{uw}\label{rem:rec:h}
As in Remark \ref{rem:rec:a}, a segment $q$ of $d_5$ obtained by a reduction of type IIIc cannot enter $N_a$ as an arc of type D. We will show later (Lemma \ref{lem:one:new:inters}) that in such a case, arcs that follow/precede $q$ in $N_a$ must intersect $b$. Hence Lemma \ref{lem:d:eq:c3} remains valid with $d_3$ replaced by $d_5$ (arcs modified by a reduction of type IIIc can not satisfy assumptions of that lemma).
\end{uw}
\begin{uw} \label{Uw:Red:join:cascade}
An arc $s$ of $d_4\cap N_a$ can be involved in multiple reductions of type IIIc (Figure \ref{r17}), that is, the arcs that start at the end-points of $s$ can be involved in several reductions of type IIIc.
\begin{figure}[h]
 \begin{center}
\fig{0.95}{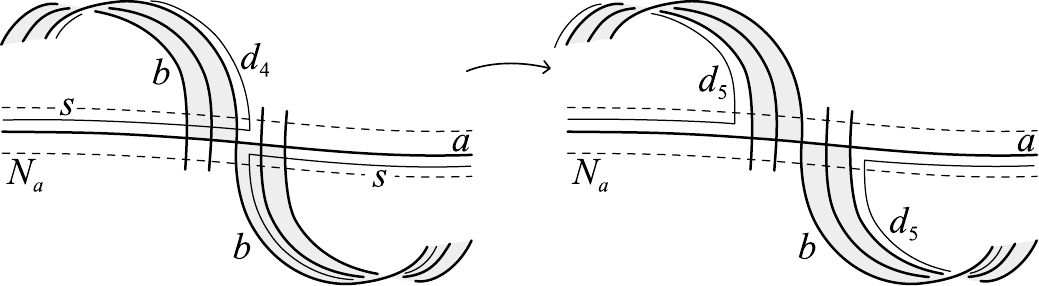}
 \caption{Cascade reductions of type IIIc.}\label{r17} %
 \end{center}
 \end{figure}
However, all these reductions can change the initial and the terminal rectangles $r_1,r_2$ of $s$ only to the rectangles joinable to $r_1$ and $r_2$, respectively (see Section \ref{sec:joinable}).
\end{uw}

\begin{lem}\label{lem:one:new:inters}
Let $s'$ be an arc of $c\cap N_a$ of type A, B or C, and let $s$ be the arc of $d_5\cap N_a$ which corresponds to $s'$. Then $s$ is parallel to $a$ in at least one rectangle of $N_{a\cap b}$.
\end{lem}
\begin{proof}
Let $s''$ be the arc of $d_3$ that corresponds to $s'$ (that is, $s''$ is obtained from $t_a^k(s')$ by reductions of types I, II, IIIa), and let $p$ and $q$ be segments of $b$ that correspond to the segments of $d_5$ that start at the endpoints of $s$.

Assume first that $s'$ is an arc of type C. Arcs of this type do not allow reductions of type IIIb (Lemma \ref{lem:ends:ab}). Hence, the only type of reductions that could decrease the number of rectangles in which $s''$ is parallel to $a$ is the reduction of type IIIc. If we assume that no rectangles of $N_{a\cap b}$ exist in which $s$ is parallel to $a$, then each double segment of $b$ must contain a segment joinable either to $p$ or $q$ (Remark \ref{Uw:Red:join:cascade}). However, this implies that $a$ and $b$ are in the special position (S1), which contradicts our assumption that $\{a,b\}$ is not special. 

If $s'$ is an arc of type A or B, then the situation is completely analogous. If $s'$ is of type A, then $s''$ can admit on one side one reduction of type IIIb (Lemma \ref{uw:Rig:r3b:obst2}), and if $s'$ is of type B, then $s''$ can admit reductions of type IIIb on both sides (one on each side -- Remark \ref{uw:red3b:cont1}). Hence, the assumption that no rectangle of $N_{a\cap b}$ exists in which $s$ is parallel to $a$ implies that $a$ and $b$ are in the special position (S1).
\end{proof}

\emph{Strong winding of $d_5$.}
We need to show that for each three rectangles $r_1,r,r_2$ of $N_{a\cap b}$, which are consecutive along $a$, an arc of $d_5$ exists, which is parallel to $a$ in $r_1,r,r_2$.
Without loss of generality, we can assume that the configuration of rectangles is as in Figure \ref{r07a}, that is, $r_2$ is on the right of $r$.
\begin{figure}[h]
 \begin{center}
\fig{0.52}{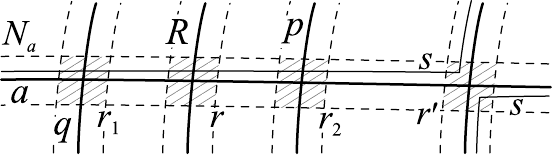}
 \caption{Intersection of $t_a^k(s')$ with rectangles of $N_{a\cap b}$.}\label{r07a} %
 \end{center}
 \end{figure}
Let $p$ and $q$ be segments of $b$ such that $p$ goes up from $r_2$ and $q$ goes down from $r_1$, and let $R$ be the double segment of $b$ that corresponds to $r$.

Let $r'$ be any rectangle of $N_{a\cap b}$ that is different from $r_1,r$ and $r_2$ (such rectangles exist given that $I(a,b)\geq  4$), and let $s'$ be an arc of $c\cap N_{a\cap b}$ that is rigid on both sides of $r'$ (it exists because $c$ winds strongly around $b$). If $s$ is an arc of $d_3$ that corresponds to $s'$, then $s$ does not allow reductions of type IIIb on either side of $r'$. Hence, it can only be reduced by reductions of type IIIc. By Remark \ref{Uw:Red:join:cascade}, if we assume that $s$ can be reduced so that the corresponding arc of $d_5$ is not parallel to $a$ in $r_1$ or $r_2$, then one of the segments of $b$ that start in $r'$ must be joinable to either $p$ or $q$. 

From the above analysis, if we assume that each arc $s'$ of $c\cap N_{a\cap b}$ leads to an arc $s$ of $d_5\cap N_a$ that is not parallel to $a$ in either $r_1$ or $r_2$, then each double segment of $b$ different from $R$ is joinable either to $p$ or $q$. Moreover, the triple $\{p,q,R\}$ is positively oriented (see Section \ref{sec:Rigid}). Hence, we are in the special position (S2) or (S3), which is a contradiction.

\emph{Bigons formed by $d_5$ and $b$.}
Let us prove that $d_5$ and $b$ are in the minimal position so they do not form any bigon. Suppose on the 
contrary that $\Delta$ is a bigon with vertices $A$ and $B$ bounded by arcs $p$ and $q$ of $d_5$ and $b$, 
respectively. By taking the inner most bigon, we can assume that the interior of $\Delta$ is disjoint from 
$d_5\cup b$. 

Since $d_5$ winds around $a$, each rectangle of $N_{a\cap b}$ contains an intersection point 
of $d_5$ and $b$. Hence, $q$ is either an arc of $b$ in a single rectangle $r_{A,B}$ in $N_{a\cap b}$,
or $q$ is a segment of $b$ that connects two different rectangles $r_A$ and $r_B$ of $N_{a\cap b}$.
In the second case, $d_5$ would admit a reduction of type II or III (depending on whether $q$ is 
two-sided or not). Hence, we concentrate on the first possibility. If $\Delta$ is contained in $N_a$, then 
$d_5$ admits a reduction of type I, which is not possible. Hence $\Delta$ is not contained in $N_a$.

Let $p'$ be the subarc of $p$, which is obtained from $p$ by removing the arcs contained in $N_a$ that connect $A$ and $B$ with the boundary of $N_a$ (Figure \ref{r16a}). 
\begin{figure}[h]
 \begin{center}
\fig{0.76}{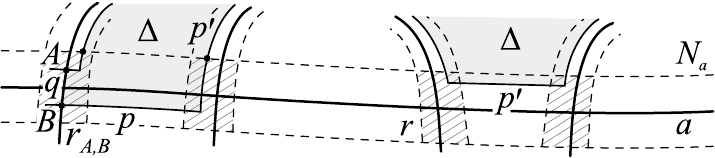}
 \caption{Bigon $\Delta$ between $b$ and $d_5$.}\label{r16a} %
 \end{center}
 \end{figure}
By Lemma \ref{lem:d:eq:c3} and Remark \ref{rem:rec:h}, the arc $p'$ of $d_5$ is in fact an arc of $c_3$. Hence, the existence of $\Delta$ contradicts Lemma \ref{lem:ExtBigon1}
 

\emph{Counting intersection points between $d_5$ and $b$.} 
To finish the proof, we need to show that $I(d_5,b)>I(d_5,a)$. The idea is to show that associated  
intersection points of $d_5\cap b$ exist for each intersection point of $d_5\cap a$.

These arcs of $d\cap N_a$ that intersect $a$ have a one-to-one correspondence to arcs of $c\cap N_a$ 
that intersect $a$, hence to arcs of types A--C.
Moreover, all reductions we performed on $d$ preserved this bijection because during the reductions 
we did not create any new arcs that intersect $a$, and we did not remove any of the existing ones. 

By Lemma \ref{lem:one:new:inters}, each arc $s$ of $d_5\cap N_a$ that corresponds to an arc $s'$ of $c\cap N_a$ of type A--C is parallel to $a$ in at least one rectangle of $N_{a\cap b}$. Hence, it intersects $b$ at least once. Moreover, because $d_5$ winds strongly around $a$, some arcs $s$ of $d_5\cap N_a$ intersect $b$ in at least 3 points. Therefore $I(d_5,b)>I(d_5,a)$. 
\end{proof}
 \section{Weak rigidity}\label{sec:weak}%
The proof of Proposition \ref{Prop:main:g3} is based on the notion of strong winding: for each three rectangles $r_1,r_2,r_3$ that are consecutive on $a$, an arc of $c\cap N_a$ exists, which leads to an arc of $d_5\cap N_a$ parallel to $a$ in $r_1,r_2$, and $r_3$.
This assertion can fail in special cases (S1)--(S3).
\begin{ex}\label{ex:spec:1}
 Let $a$ and $b$ be two circles that are in the special position (S1), as shown in Figure \ref{r21a}(i). 
  \begin{figure}[h]
\begin{center}
\fig{0.99}{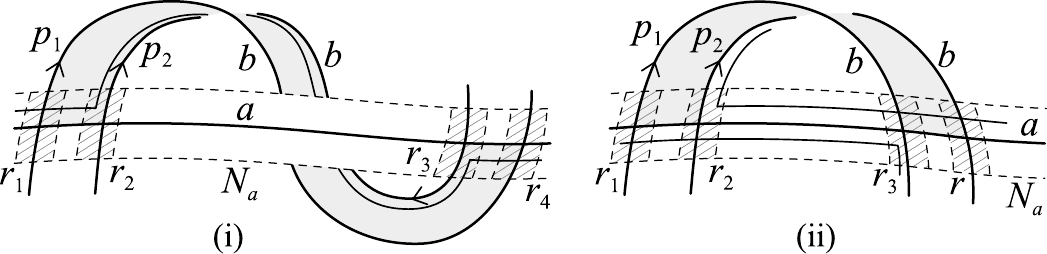}
\caption{Possible failure in the proof of Proposition \ref{Prop:main:g3} -- Examples \ref{ex:spec:1} and \ref{ex:spec:2}.}\label{r21a} %
\end{center}
\end{figure}
 Segments $p_1$ and $p_2$ are adjacent, thus the arcs of $d$, which are obtained from arcs of $c$ of type A in the rectangle $r_2$, may admit a reduction of type IIIc. Hence, none of these arcs may be parallel to $a$ in $r_1$. Similarly, all arcs of $d$ that are obtained from arcs of $c$ of type A in $r_3$ may reduce by reductions of type IIIc. Hence, none of these arcs may be parallel to $a$ in $r_4$. Therefore, a possibility that no arc of $d_5\cap N_a$ which is parallel to $a$ in $r_1$ and $r_4$ exists (hence, $d_5$ does not wind strongly around $a$).
\end{ex}
\begin{ex}\label{ex:spec:2}
 Let $a$ and $b$ be two circles that are in the special position (S3) as shown in Figure \ref{r21a}(ii). Segments $p_1$ and $p_2$ are adjacent. Thus, the arcs of $c\cap N_a$ which are of type A in $r_1$ or $r_2$ give arcs of $d$ that can be reduced so that the resulting arcs of $d_5\cap N_a$ are not parallel to $a$ in $r_1$. Similarly, arcs of $c\cap N_a$ that are of type A in $r_3$ or $r$ yield arcs of $d_5\cap N_a$ that are not parallel to $a$ in $r_3$. Therefore, a possibility is that no arc of $d_5\cap N_a$ that is parallel to $a$ in $r_1$ and $r_3$ exists (hence, $d_5$ does not wind strongly around $a$).
%
%
\end{ex}
To overcome the abovementioned problems, we redefine the rigidity of arcs slightly. 

Let $r_1,r_2,r_3,r_4,r_5,r_6$ be rectangles of $N_{a\cap b}$ that are consecutive vertices of an exterior hexagon $\Delta$ (consecutive here means consecutive along $\partial \Delta$). Suppose also that $q$ is an arc of a circle $c\in {\cal C}$ as in Figure \ref{r00c}(i), that is, $q$ is of type A in $r_1$ and $r_4$, $q$ is of type D between $r_2$ and $r_3$, and the segment of $b$ that connects $r_5$ and $r_6$ is one-sided. In such a case we say that $q$ is a \emph{one-sided boundary 3-segment} of $\Delta$.
\begin{uw}\label{rem:concordance}
 Let $\Delta$ be an exterior hexagon 
 and let $q$ be an arc of $c$, which leads to an arc $q'$ of $d_3$ and allows a reduction of type IIIb across $\Delta$ [Figure \ref{r00c}(i)]. According to the definition of the reduction of type IIIb, if $q$ enters $N_a$ as an arc of type A (on both ends), then $q$ is a one-sided boundary 3-segment of $\Delta$.
\end{uw}
Let $q$ be an arc of $c\in{\cal C}$ that is parallel to $b$ in a rectangle $r_1$ of $N_{a\cap b}$. Some orientation of $q$ is fixed, and $q$ is followed to the rectangles $r_2,r_3,r_4$ of $N_{a\cap b}$ following $r_1$. We say that $q$ is \emph{weakly rigid} in $r_1$ with respect to $b$ if either 
\begin{itemize}
 \item $q$ is rigid with respect to $b$ in $r_1$, or
 \item $q$ does not intersect $a$ in $r_2$ and $r_3$, $q$ is parallel to $b$ in $r_4$, and $q$ is not an one-sided boundary 3-segment of an exterior hexagon (Figure \ref{r15d}).
\end{itemize}
\begin{figure}[h]
\begin{center}
\fig{0.83}{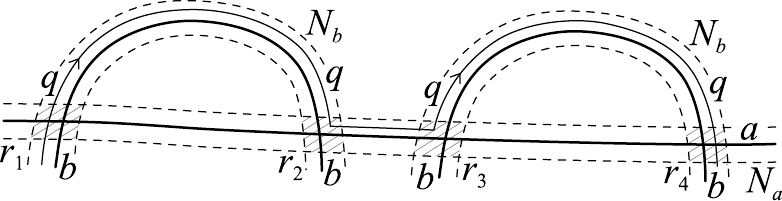}
\caption{Arc $q$ of $c$ which is weakly rigid in a rectangle $r_1$ of $N_{a\cap b}$.}\label{r15d} %
\end{center}
\end{figure}
If $c$ winds around $b$, then equivalently (from the perspective of $q$ intersecting $N_a$), $q$ is of type A in $r_1$ and then either $q$ is of type A in $r_2$, or $q$ is of type D between $r_2$ and $r_3$, $q$ is of type A in $r_4$, and $q$ is not a one-sided boundary 3-segment of an exterior hexagon.

The following lemma shows that from the point of view of the proof of Proposition \ref{Prop:main:g3}, weakly rigid arcs are as good as rigid arcs.
\begin{lem}\label{s3:lem1}
 Suppose that $c\in{\cal C}$ winds around $b$ and let $q$ be an oriented arc of $c$ which is weakly rigid in a rectangle $r_1$ of $N_{a\cap b}$. Then $t_a^k(q)$ does not admit a reduction of type IIIb.
\end{lem}
\begin{proof}
 Suppose that $q$ after leaving $r_1$ goes around the boundary of an exterior $n$-gon $\Delta$ and then enters $N_a$ in a rectangle $r_4$ of $N_{a\cap b}$ as an arc of type A (Figure \ref{r00c}(i)).
 \begin{figure}[h]
\begin{center}
\fig{0.99}{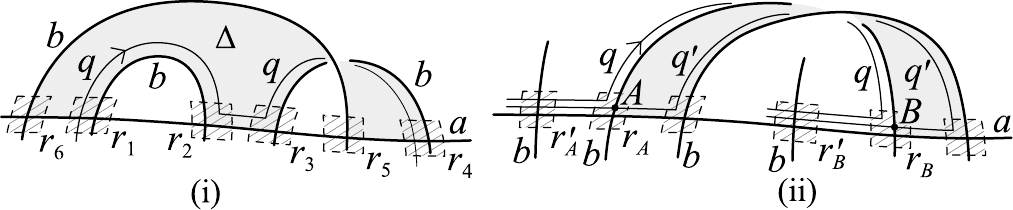}
\caption{Reductions of types IIIb and IIIc.}\label{r00c} %
\end{center}
\end{figure}
 If $q$ leads to an arc that admits a reduction of type IIIb, then $q$ between $r_1$ and $r_4$ must follow $\frac{n}{2}-2$ rectangles of $N_{a\bez b}$. Hence, it must $\frac{n}{2}-2$ times enter $N_a$ as an arc of type D. However, by the definition of weak rigidity, $q$ can intersect $N_a$ as an arc of type D only once. Hence, $n=6$ and by Remark \ref{rem:concordance}, $q$ is a one-sided boundary 3-segment of $\Delta$. However, this contradicts the assumption that $q$ is weakly rigid in $r_1$.
\end{proof}
The next lemma shows that the reductions of types IIIb and IIIc produce arcs that are good candidates for weakly rigid arcs. 

Let $c,d_3,d_5$  be as in the proof of Proposition \ref{Prop:main:g3}, that is, $c\in {\cal C}$ winds around $b$, $d_3$ is obtained from $t_a^k(c)$ by all possible reductions of types I--IIIa, and $d_5$ is obtained from $d_3$ by all possible reductions of types IIIb and IIIc. Assume also that $c$ is such that the statements of Lemma \ref{lem:ends:ab} and Remark \ref{prop3b3c} holds true.
\begin{lem}\label{lem:spec:same:side:prod:weak}
 Suppose that each arc of $c\cap N_a$ of types A--C gives an arc of $d_5\cap N_a$ that is parallel to $a$ in at least one rectangle of $N_{a\cap b}$. Let $q$ be an oriented arc of $d_5$ which starts in a rectangle $t$ of $N_{a\cap b}$ as an arc parallel to $a$, then it follows $a$ to the next rectangle of $N_{a\cap b}$ and then it follows an arc obtained from an arc of $d_3\bez N_a$ by reductions of types IIIb and/or IIIc [Figure \ref{r00c}(ii)]. If $q$ is not a one-sided boundary 3-segment of an exterior hexagon, then $q$ is weakly rigid in $t$ with respect to $a$. 
\end{lem}
\begin{proof}
 Let $q'$ be an arc of $d_3$, which can be reduced to the arc $q$ of $d_5$ by reductions of types IIIb and/or IIIc. Assume that the endpoints of $q'$ are $A,B\in d_3\cap b$ and let $r_A,r_B$ be rectangles of $N_{a\cap b}$ that contain $A$ and $B$, respectively. By the general properties of reductions of type III (see Lemma \ref{lem:ends:ab} and Remark \ref{prop3b3c}), both ends of $q'$ were formed from the arcs of $c$ intersecting $a$ (hence not arcs of type D). Therefore, by our assumption, both ends of $q$ after entering $N_a$ must run parallel to $a$ in at least one rectangle of $N_{a\cap b}$. Orient $q$ from $r_A$ to $r_B$ and extend $q$ on both ends so that it starts and terminates in rectangles $r_A'$ and $r_B'$ of $N_{a\cap b}$, which respectively precede and follow $r_A$ and $r_B$ along $q$.
 From the point of view of $q$ intersecting $N_b$, this arc is of type A in $r_A'$, then it intersects $N_b$ as an arc of type D, and then it is again an arc of type A (in $r_B'$). This finding means that if this arc is not a one-sided boundary 3-segment of an exterior hexagon, then $q$ is weakly rigid in $r_A'$. 
\end{proof}

 We say that $c\in{\cal C}$ is \emph{weakly rigid} with respect to $b$ if each arc $p$ of $c$ that is parallel to $b$ in a rectangle $r$ of $N_{a\cap b}$ is weakly rigid with some choice of orientation.
 \begin{uw}\label{Rem:weak:subs3}
  If $c$ is weakly rigid with respect to $b$, then by Lemma \ref{s3:lem1}, $c$ satisfies the statement of Lemma \ref{uw:Rig:r3b:obst2}. Hence, the assumption that $c$ is weakly rigid with respect to $b$ serves as a replacement for Lemma \ref{uw:Rig:r3b:obst2}.
 \end{uw}
 We say that $c\in{\cal C}$ \emph{winds weakly} around $b$ if for every rectangle $r$ of $N_{a\cap b}$ and each choice of orientation for the collection $P$ of arcs of $c$ that are parallel to $b$ in $r$ a weakly rigid arc exists in $P$.
\begin{uw}\label{Rem:weak:subs}
 In the proof of Proposition \ref{Prop:main:g3}, the assumption that $c$ winds strongly around $b$ was used to conclude that in each rectangle $r$ of $N_{a\cap b}$ an arc $s$ of $c$ exists, which is of type A in $r$ and which leads to an arc of $d_3\cap N_a$ that does not allow a reduction of type IIIb on either side of $r$ (Remark \ref{uw:Rig:r3b:obst1}). By Lemma \ref{s3:lem1}, weak winding provides a slightly weaker conclusion: for each $a$-side of $r$, an arc $s$ of $c$ exists, which is of type A in $r$ and which yields an arc of $d_3$ that does not allow a reduction of type IIIb on the chosen side of $r$ (for such $s$, choose an arc of $c$ that is weakly rigid on the chosen side of $r$).
\end{uw}
\begin{uw}\label{Rem:weak:subs2}
 Remark \ref{Rem:weak:subs} implies that the proof of Lemma \ref{lem:ends:ab} remains valid if we replace the notion of strong winding with the notion of weak winding. In fact, that proof was based on the fact that if $r$ is a rectangle $r$ of $N_{a\cap b}$, then for each $a$-side of $r$, an arc $s$ of $c$ exists, which is of type A in $r$ and which yields an arc of $d_3$ that does not allow a reduction of type IIIb on the chosen side of $r$. 
\end{uw}
Let $\widehat{X}_a$ be the set of isotopy classes of circles $c$ in $N$, which satisfy the following conditions
\begin{enumerate}
 \item $c\in {\cal C}$,
 \item $I(c,a)=|c\cap a|$, $I(c,b)=|c\cap b|$,
 \item $I(c,a)<I(c,b)$,
 \item $c$ is weakly rigid with respect to $a$,
 \item $c$ winds weakly around $a$.
\end{enumerate}
Similarly, we define $\widehat{X}_b$ by requiring (1)--(2) above and additionally
\begin{enumerate}
 \item[(3')] $I(c,b)<I(c,a)$,
 \item[(4')] $c$ is weakly rigid with respect to $b$,
 \item[(5')] $c$ winds weakly around $b$.
\end{enumerate}
\begin{lem}\label{Lem:two:NewVer}
 Let $c\in \widehat{X}_b$ and assume that $a$ and $b$ are not in the special position (S1), (S2) nor (S3). Let $d_5$ be as in the proof of Proposition \ref{Prop:main:g3} and let $s'$ be an arc of $c\cap N_a$ of type A, B or C. If $s$ is the arc of $d_5\cap N_a$ which corresponds to $s'$, then $s$ is parallel to $a$ in at least two rectangles of $N_{a\cap b}$.
\end{lem}
\begin{proof}
Let $s''$ be the arc of $d_3$ that corresponds to $s'$ (that is $s''$ is obtained from $t_a^k(s')$ by reductions of types I, II, IIIa), and let $p$ and $q$ be segments of $b$ that correspond to the segments of $d_5$ that start at the endpoints of $s$.

Assume first that $s'$ is an arc of type C. Arcs of this type do not allow reductions of type IIIb (Lemma \ref{lem:ends:ab} and Remark \ref{Rem:weak:subs2}). Hence, the only type of reduction that could decrease the number of rectangles in which $s''$ is parallel to $a$ is the reduction of type IIIc. If we assume that no rectangles of $N_{a\cap b}$ exist in which $s$ is parallel to $a$, then each double segment of $b$ must contain a segment joinable either to $p$ or $q$ (Remark \ref{Uw:Red:join:cascade}). But this implies that $a$ and $b$ are in the special position (S1), which contradicts our assumption. 
Analogously, if we assume that $s$ is parallel to $a$ in only one rectangle $r$ of $N_{a\cap b}$ and $R$ is the double segment of $b$ that corresponds to $r$, then each double segment of $b$ that is different from $R$ contains an oriented segment joinable to $p$ or $q$. Moreover, the triple $\{p,q,R\}$ must be positively oriented. Hence, 
$a$ and $b$ are in the special position (S2) or (S3), which again is a contradiction.

If $s'$ is an arc of type A or B, then the situation is completely analogous. If $s'$ is of type A, then $s''$ can admit on one side one reduction of type IIIb (Lemma \ref{uw:Rig:r3b:obst2},  Remarks \ref{Rem:weak:subs3} and \ref{uw:red3b:cont1}), and if $s'$ is of type B, then $s''$ can admit reductions of type IIIb on both sides (one on each side -- Remark \ref{uw:red3b:cont1}). Hence, the assumption that no rectangle of $N_{a\cap b}$ exists in which $s$ is parallel to $a$ implies that $a$ and $b$ are in the special position (S1), and if we assume that $s$ is parallel to $a$ in only one rectangle of $N_{a\cap b}$, then $a$ and $b$ are in the special position (S2) or (S3).
\end{proof}
\begin{lem}\label{Lem:WeakWin:NewVer}
 Let $c\in \widehat{X}_b$ and assume that $a$ and $b$ are not in the special position (S1), (S2) nor (S3). If $d_5$ is as in the proof of Proposition \ref{Prop:main:g3}, then $d_5$ winds weakly around $a$.
\end{lem}
\begin{proof}
 We need to show that for each rectangle $r$ of $N_{a\cap b}$ and for each orientation of arcs parallel to $a$ in $r$, an arc of $d_5$ exists, which is parallel to $a$ in $r$ and is weakly rigid with respect to the chosen orientation of arcs in $r$. 

Fix $r$ and let $r_1$ be the rectangle of $N_{a\cap b}$ following $r$ along $a$ with respect to the chosen orientation of arcs in $r$. Without loss of generality we can assume that the configuration of rectangles is as in Figure \ref{r07d}, that is $r_1$ is on the right of $r$. 
\begin{figure}[h]
 \begin{center}
\fig{0.76}{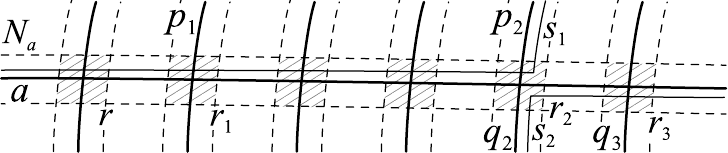}
 \caption{Intersection of $t_a^k(s')$ with rectangles of $N_{a\cap b}$.}\label{r07d} %
 \end{center}
 \end{figure}
Next, follow $a$, with the chosen orientation of arcs in $r$, to the first rectangle $r_2$ of $N_{a\cap b}$  such that the top oriented segments of $b$ in $r_1$ and $r_2$ are not joinable
($r_2$ exists by Proposition \ref{prop:sing:dbl:seg}). If $r_2=r$, then each double segment of $b$ contains a segment joinable to the top segment of $r_1$ or to the bottom segment of $r$. Hence, $a$ and $b$ are in the special position (S1). Therefore, $r_2\neq r$.

Let $r_3$ be the rectangle of $N_{a\cap b}$ that follows $r_2$. If $r_3=r$ then $a$ and $b$ are in a special position [if the top segment of $r_1$ is joinable to the top segment of $r$ or if the bottom segment of $r_2$ is joinable with the bottom segment of $r$, then we are in the special position (S1); otherwise, we are in the special position (S2) or (S3)]. Hence, $r_3\neq r$. 

Let $s'$ be an arc of $c\cap N_a$ that is parallel to $b$ in $r_2$ and that is weakly rigid on the top side of $r_2$ (it exists because $c$ winds weakly around $b$). Under this assumption the segment $s_1$ of $t_a^k(c)$ that starts at the top endpoint of $s'$ does not allow a reduction of type IIIb (Lemma \ref{s3:lem1}). It may allow reductions of type IIIc, but they cannot reach $r_1$, because $r_2$ is not joinable to $r_1$. 

Now, we concentrate on the segment $s_2$ of $t_a^k(c)$ starting at the bottom endpoint of $s'$. Let $p_i$ and $q_i$, for $i=1,2,3$, be segments of $b$, which go up and down from $r_i$, respectively.
Segment $s_2$ may admit a reduction of type IIIb and some reductions of type IIIc, but if these reductions reach $r$, then either all double segments of $b$ contain a segment joinable to $p_1$ or $q_2$ (this case happens when $s_2$ can be reduced by reductions of type IIIc), or only one double segment of $b$  (corresponding to $r_2$) exists, which does not contain a segment joinable to $p_1$ or $q_3$ (this happens when $s_2$ is reduced by a reduction of type IIIb and then by reductions of type IIIc). According to our assumption that $a$ and $b$ are not special, such a situation is not possible. Hence, the reductions on $s_2$ cannot reach $r$, and as a consequence, the arc $s$ of $d_5$, that corresponds to $s'$ is parallel to $a$ in $r$ and $r_1$.
\end{proof}
 \section{The special cases (S1) and (S2)}\label{sec:specA}
The common feature of cases (S1) and (S2) is the existence of oriented segments $p$ and $q$ starting on different sides of $a$, such that each double segment (or each double segment except one in case (S2)) contains a segment joinable to $p$ or $q$. In case (S1) this implies the possibility that some arcs of $c\cap N_a$ of types A--C may lead to arcs of $d_5\cap N_a$, which does not intersect $b$ (see the proof of Lemma \ref{lem:one:new:inters}). We will show below (Lemma \ref{Lem:s13:inner}) that such reductions are not possible.

The second problem in cases (S1) and (S2) is that in these cases, $d_5$ may not wind strongly around $a$. We deal with this problem by replacing the notion of strong winding with the notion of weak winding defined in the previous section.

%
%
%
%
%
%
%
%
%
\begin{lem}\label{MobRest}
 Let $\map{\pi}{M}{S^1}$ be a bundle over $S^1$ with fiber $I=[0,1]$ and which is homeomorphic to a \Mob. If a simple oriented arc $c$ in $M$ is monotone with respect to the fixed orientation of $S^1$ and has endpoints in $\partial M$, then $c$ intersects every fiber in at most two points.
\end{lem}
\begin{proof}
 If $c$ intersects some fiber in at least 3 points, then $c$ winds infinitely many times around the core of $M$ [Figure \ref{r15}(i)].
  \begin{figure}[h]
 \begin{center}
\fig{0.7}{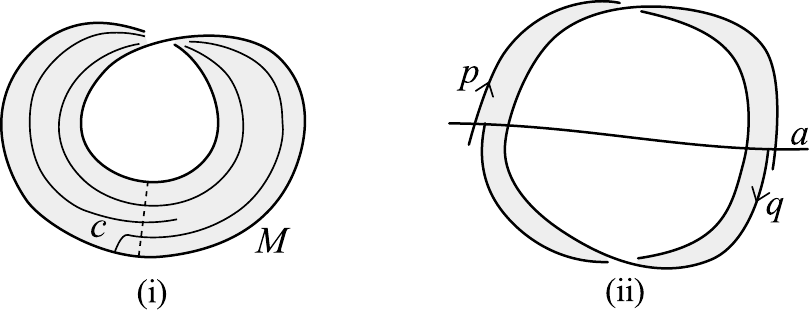}
 \caption{Impossible configurations of arcs -- Lemmas \ref{MobRest} and \ref{lem:spec:diff:sides}.}\label{r15} %
 \end{center}
 \end{figure}
\end{proof}
\begin{lem}\label{lem:spec:diff:sides}
 Let $a$ and $b$ be two generic two-sided circles in $N$ such that $I(a,b)\geq 4$, and assume that 
 oriented segments $p$ and $q$ of $b$ exist, which start on different sides of $a$, such that each
 double segment of $b$ contains an oriented segment joinable to $p$ or $q$ (that is $a$ and $b$ are in the special position (S1)).  
 Then $p$ is joinable to $-q$.
\end{lem}
\begin{proof}
 Suppose first that $p$ starts and terminates on the same side of $a$. In such a case, by Proposition 
 \ref{prop:prop:segments}, segments that start at the 
 terminal points of segments joinable to $p$ must be joinable to $q$, and vice versa.
 However, this implies that $b$ is a circle on 
 an annulus [that is the union of twisted rectangles given by adjacency: see Figure \ref{r15}(ii)].
 This contradicts the assumption that $I(a,b)\geq 4$.
 
 Hence, $p$ starts and terminates on two different sides of $a$. But then, the segment starting at the terminal point
 of $p$ cannot be joinable to $q$ (because $p$ and $q$ begin on two different sides of $a$), and it cannot
 be joinable to $p$ (by Proposition \ref{prop:prop:segments}). Hence, $-p$ must be joinable to $q$.
\end{proof}
\begin{uw}\label{rem:S1}
 Let $a,b,p$ be as in the above lemma and let $p_1,p_2,\ldots,p_n$ be all oriented segments of $b$ joinable to $p$. The union of adjacency disks between $p_1,p_2,\ldots,p_n$ provides a rectangle $\Gamma$. We can assume that $p$ goes up from $a$ and that $p_1$ and $p_n$ are on the boundary of $\Gamma$, where $p_n$ leaves $N_a$ to the left of $p_1$. By Lemma \ref{lem:spec:diff:sides}, $q$ is equal to one of the segments: $-p_1,-p_2,\ldots,-p_n$. Hence, without loss of generality, we can assume that $q=-p_1$. The other two boundary arcs of $\Gamma$ are arcs $a_1,a_2$ of $a$. Observe that by Proposition \ref{prop:prop:segments}, $a_1\cap a_2=\emptyset$, hence $a\bez(a_1\cup a_2)$ consists of two arcs and the configuration of $a$ and $b$ is as in the left part of Figure \ref{r21}. 
  \begin{figure}[h]
 \begin{center}
\fig{0.68}{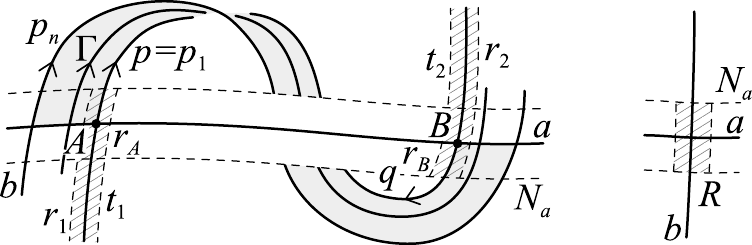}
 \caption{Configuration of arcs in cases (S1) and (S2).}\label{r21} %
 \end{center}
 \end{figure}
\end{uw}
\begin{lem}\label{lem:spec:same:sides}
Let $a$ and $b$ be two generic two-sided circles in $N$ such that $I(a,b)\geq 4$, and assume that 
 oriented segments $p,q$ of $b$ and a double segment $R$ of $b$ exist, such that $a$ and $b$ are not in the special position (S1), $p$ and $q$ start on different sides of $a$, $p$ starts and terminates on different sides of $a$, each double segment of $b$  different from $R$ contains an oriented segment joinable to $p$ or $q$, $\{p,q,R\}$ is positively oriented [that is $a$ and $b$ are in the special position (S2)]. Then $p$ is joinable to $-q$.
\end{lem}
\begin{proof}
Our first claim is that an oriented segment of $b$ exists, which is joinable to $p$ and such that it neither starts nor terminates in $R$ (that is neither $p$ nor $-p$ is a segment of $R$). Given that $I(a,b)\geq 4$, at least 3 intersection points of $a\cap b$ different from $R$ exist. Hence, at least two arcs $s$ and $t$ are joinable to $p$ or joinable to $q$, which do not start in $R$. If both these arcs are joinable to $p$, then at least one of them cannot terminate in $R$ (given that they terminate on the same side of $a$). Hence, our claim follows. Now assume that these two arcs are joinable to $q$. If $q$ starts and terminates on different sides of $a$, then the roles of $p$ and $q$ are symmetric, and we can prove our claim by relabeling $p$ to $q$, and vice versa. Hence, assume that $q$ starts and terminates on the same side of $a$. Let $s'$ and $t'$ be oriented segments of $b$ that follow $s$ and $t$, respectively [Figure \ref{r22}(i)].
\begin{figure}[h]
 \begin{center}
\fig{0.89}{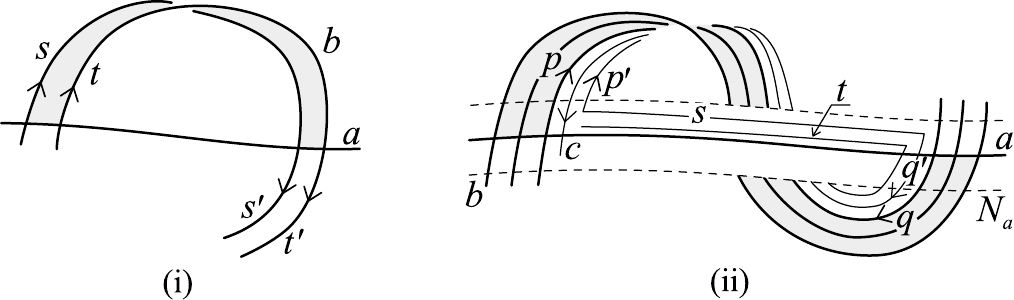}
 \caption{Configuration of arcs, Lemmas \ref{lem:spec:same:sides} and \ref{LemSpecS1c}.}\label{r22} %
 \end{center}
 \end{figure}
By Proposition \ref{prop:prop:segments}, none of $-s,-t,s',t'$ is joinable to $q$ and by our assumptions $-s$ and $-t$ are not joinable to $p$ (because $p$ and $q$ start on different sides of $a$). Hence, at least one of $s',t'$ is joinable to $p$. If both $s'$ and $t'$ are joinable to $p$, then our claim follows because only one of them can terminate in $R$. If only one of them, say $s'$, is joinable to $p$, then $R$ corresponds to the terminal point of $t$. Hence, $s'$ cannot terminate in $R$, which again proves our claim.

Assume that $s$ is a segment of $b$ joinable to $p$, which does not start nor terminate in $R$.
 Then, the segment that starts at the terminal point
 of $p$ cannot be joinable to $q$ (because $p$ and $q$ begin on two different sides of $a$), and it cannot
 be joinable to $p$ (by Proposition \ref{prop:prop:segments}). Hence, $-p$ must be joinable to $q$.
\end{proof}
\begin{uw}\label{rem:S2}
 Following the lines of the analysis conducted in Remark \ref{rem:S1}, we conclude that the configuration of $a$ and $b$ in case (S2) differs from that in (S1) only by one additional double segment $R$, which intersects $a$ in one of the arcs of $a\bez(a_1\cup a_2)$ ($a_1$ and $a_2$ are defined as in Remark \ref{rem:S1}). Hence the configuration is as in Figure \ref{r21}. (Note that the mutual position of $p,q$, and $R$ is determined by the assumption that $\{p,q,R\}$ is positively oriented.)
\end{uw}
%
%
For the rest of this section, we will use the notation introduced in Remarks \ref{rem:S1} and \ref{rem:S2}, that is $p_1=p$, $p_n$ are segments of $b$ which together with arcs $a_1,a_2$ of $a$ bound a rectangle $\Gamma$ that contains all segments of $b$ joinable to $p$, $p_1$ goes up from $a$ and leaves $N_a$ to the right of $p_n$ (Figure \ref{r21}). Moreover, assume that $p$ starts at $A$ in a rectangle $r_A$ of $N_{a\cap b}$ and it terminates in $B$ in a rectangle $r_B$. Let $r_1$ and $r_2$ be rectangles of $N_{b\bez a}$ which precede and follow $p$, respectively. 
\begin{lem}\label{rem:no:adjacency:s1}
Suppose that $a$ and $b$ are in the special position (S1) and let $u_1,\ldots,u_n$ and $v_1,\ldots,v_n$ be the oriented segments of $b$ which respectively go down/up from the initial/terminal points of $p_1,\ldots,p_n$. 
Then $u_i$ is not joinable to $u_j$ for some $i\neq j$ and $v_i$ is not joinable to $v_j$ for some $i\neq j$.
\end{lem}
\begin{proof}
If for example all the $u_i$ are mutually joinable, then $a$ and $b$ are circles on the annulus given by adjacency disks between $p_i$ and $u_i$.This contradicts the assumption $I(a,b)\geq 4$.
\end{proof}
\begin{lem}\label{rem:no:adjacency:s2}
Suppose that $a$ and $b$ are in the special position (S2) and let $u_1,\ldots,u_n$ and $v_1,\ldots,v_n$ be the oriented segments of $b$ which respectively go down/up from the initial/terminal points of $p_1,\ldots,p_n$. 
Then, the segments that constitute $R$ are joinable neither to $u_1$ nor to $v_1$. Moreover, $u_i$ is not joinable to $u_j$ for some $i\neq j$ or $v_i$ is not joinable to $v_j$ for some $i\neq j$.
\end{lem}
\begin{proof}
If a segment of $R$ is joinable to $u_1$ or $v_1$ then $a$ and $b$ are in the special position (S1), which is not possible. Next, suppose to the contrary that all the $u_i$ are mutually joinable and all the $v_i$ are mutually joinable. Given that $I(a,b)\geq 4$, $v_i=-u_j$ for some $i,j$. Hence, we can assume that each $v_i$ is joinable to $-u_1$. One of the $u_i$ must terminate in a rectangle $r$ that corresponds to $R$ and also one of the $v_i$ must terminate in $r$. However, this situation contradicts Proposition \ref{prop:prop:segments} given that $v_i$ is joinable to $-u_i$.
\end{proof}
Suppose that the component of $N\bez N_{a\cup b}$, which is determined by $p_1=p$ and an arc of $a\bez (a_1\cup a_2)$, is an exterior $n$-gon $\Delta$. Let $t_1,t_2$ be these boundary sides of $r_1$ and $r_2$, respectively, which enter $r_A$ and $r_B$ on the right of $b\cap r_A$ and $b\cap r_B$ (Figure \ref{r21}).

Now we define two \Mob s associated with $p$. The first one $M_1$ is the union of the rectangle $r_p$ of $N_{b\bez a}$ that contains $p$, rectangles $r_A,r_B$, and a single rectangle $r_{AB}$ of $N_{a\bez b}$ that connects $r_A$ and $r_B$. The second one $M_2$ is the union of $r_A, r_B, r_{AB}$ and all the rectangles that connect $r_A$ and $r_B$ along the boundary of $\Delta$.

Finally, for the rest of this section assume that $c\in{\cal C}$ and $c$ winds around $b$.
\begin{lem}\label{LemSpecS1c}
 Let $s$ be an arc of type C of $c\cap N_a$ which connects the initial points of oriented segments $p'$ and $q'$ of $c$ which run parallel to $p_1$ and $-p_1$ respectively. Then at least one of the oriented arcs of $c\cap N_a$ following $p'$ and $q'$ is an arc of type $D$ which turns to the left as it enters $N_a$.
\end{lem}
\begin{proof}
 Two possible configurations of arcs $p'$ and $q'$ exist (they can pass each other in two different ways). However, these configurations lead to the same conclusions. Hence, assume that we have the configuration shown in Figure \ref{r22}(ii). Consider the arc $t$ of $c\cap N_a$ following $p'$. After entering $N_a$, this arc must turn to the left and is either of type C or D. If $t$ was of type C, then $c$ would wind at least 3 times around the core of the \Mob\ $M_1$, which contradicts Lemma \ref{MobRest}.
 Hence, $t$ is an arc of type D.
\end{proof}
Let $t$ be a common boundary of a rectangle $r$ in $N_{a\bez b}\cup N_{b\bez a}$ and an exterior $n$-gon $\Delta$. We say that an arc $q$ of $c\cap r$ is a \emph{bounding segment} for $\Delta$ (with respect to $t$) if no arcs of $c$ exist between $t$ and $q$ in $r$. Clearly, $c\cap r$ is either empty or it contains exactly one bounding segment for $\Delta$ with respect to $t$. 
\begin{lem}\label{LemSpecS1a}
Let $r_p$ be the rectangle in $N_{b\bez a}$ that contains $p$.
 If $s_1$ is an arc of $c$ that passes through $r_1,r_A,r_p$ and is of type A in $r_A$, and 
 $s_2$ is an arc of $c$ which passes through $r_2,r_B,r_p$ and is of type A in $r_B$, then either $s_1\cap r_1$ is not a bounding segment for $\Delta$ with respect to $t_1$ or $s_2\cap r_2$ is not a bounding segment for $\Delta$ with respect to $t_2$.
\end{lem}
\begin{proof}
 Two possible configurations of arcs $s_1$ and $s_2$ exist (they can pass each other in two different ways). However, these configurations lead to the same conclusions. Hence, assume that we have the configuration shown in Figure \ref{r23}(i). 
 \begin{figure}[h]
 \begin{center}
\fig{0.99}{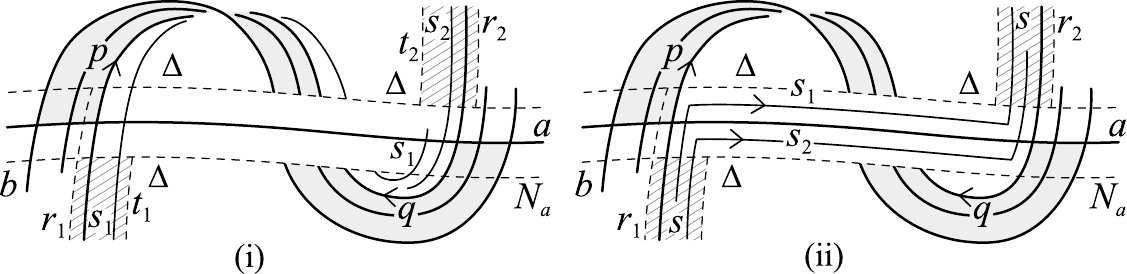}
 \caption{Configuration of arcs, Lemmas \ref{LemSpecS1a}, \ref{LemSpecS1b} and \ref{Lem:s13:inner}.}\label{r23} %
 \end{center}
 \end{figure}
 
 Consider the arc $s$ of $c\cap N_a$ following $s_1$. If $s$ was an arc of type C, then $c$ would wind at least 3 times along the core of the \Mob\ $M_1$ and this would contradict Lemma \ref{MobRest}. Hence, $s$ is either of type A or D. In the former case, $s_2$ is not a bounding segment for $\Delta$ with respect to $t_2$, and in the later case, $s_1$ is not a bounding segment for $\Delta$ with respect to $t_1$.
\end{proof}
If $s$ is an arc of $c\in{\cal C}$ with endpoints in $r_A$ and $r_B$ that connects $r_A$ and $r_B$ in $M_2\bez r_p$, then we say that $s$ is a \emph{long bounding segment for $\Delta$}. 
\begin{lem}\label{LemSpecS1b}
Let $s_1,s_2$ be oriented arcs of $c\cap N_a$ of type $B$ such that the arc $s$ of $c$ that starts at the terminal point of $s_1$ and terminates at the starting point of $s_2$ is a long bounding segment for $\Delta$ (Figure \ref{r23}(ii)). Then the arcs of $c$ which precede $s_1$ and follow $s_2$ are not long bounding segments for $\Delta$.
\end{lem}
\begin{proof}
 If for example the arc that precedes $s_1$ was a long bounding segment for $\Delta$, then $c$ would be a curve in the \Mob\ $M_2$ which winds at least 3 times along the core of $M_2$; this contradicts Lemma \ref{MobRest}.
\end{proof}
\begin{lem}\label{LemSpecS1ab}
Assume that an arc $s$ of $c\cap N_a$ of type B exists, which starts in $r_A$ and terminates in $r_B$. 
If $s_1$ [or $s_2$] is an arc of $c$ that is of type A in $r_A$ [or $r_B$], then $s_1$ [or $s_2$] is not a bounding segment for $\Delta$ with respect to $t_1$ [or $t_2$].
\end{lem}
\begin{proof}
 Given that $c$ cannot intersect itself, $s_1$ and $s_2$ must enter $N_a$ on the left of $s$ [our point of view here is along $s_1/s_2$ towards $r_A/r_B$; see Figure \ref{r24}(i)].
 \begin{figure}[h]
 \begin{center}
\fig{0.99}{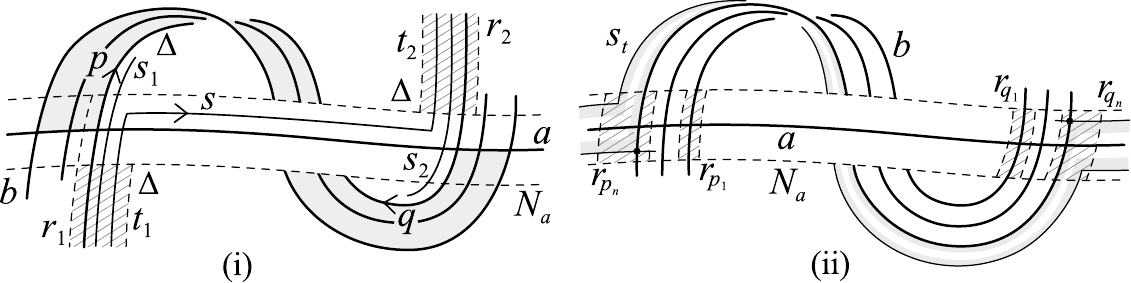}
 \caption{Configuration of arcs, Lemma \ref{LemSpecS1ab} and Proposition \ref{Main:prop:s1}.}\label{r24} %
 \end{center}
 \end{figure}
\end{proof}
Now we are ready to adopt the proof of Proposition \ref{Prop:main:g3} to the special cases (S1) and (S2). However, as we observed in Example \ref{ex:spec:1}, the notion of rigidity is too strong in these cases. Hence, we replace it with the notion of weak rigidity (see Section \ref{sec:weak}).

\begin{prop}\label{Main:prop:s1}
 Let $a$ and $b$ be two generic two-sided circles in $N$ such that $I(a,b)\geq 4$ and $\{a,b\}$ is not in the special position (S3). Then, for any integer $k\neq 0$, we have
 \[t_a^k(\widehat{X}_b)\podz \widehat{X}_a\quad\text{and}\quad t_b^k(\widehat{X}_a)\podz \widehat{X}_b.\]
\end{prop}
\begin{proof}
As we observed in Remark \ref{Rem:weak:subs2} the assumption about weak winding of $c$ around $b$ suffices to prove the statement of Lemma \ref{lem:ends:ab}. As a substitute for Remark \ref{uw:Rig:r3b:obst1}, we have slightly weaker Remark \ref{Rem:weak:subs}. Moreover, as observed in Remark \ref{Rem:weak:subs3}, the assumption that $c$ is weakly rigid with respect to $b$ serves as a replacement for Lemma \ref{uw:Rig:r3b:obst2}. Hence, in most parts, we can copy verbatim the proof of Proposition \ref{Prop:main:g3}, yet there are some differences, which we study in detail below.

Suppose first that $a$ and $b$ are not in the special position (S1) or (S2).
In such a case, by Lemmas \ref{Lem:two:NewVer} and \ref{Lem:WeakWin:NewVer}, $d_5$ is weakly rigid with respect to $a$ and $d_5$ winds weakly around $a$. Moreover, Lemma \ref{Lem:two:NewVer} implies that $I(d_5,b)>I(d_5,a)$.
%
%

Therefore, for the rest of this section, we can concentrate on circles $a$ and $b$, which are in the special position (S1) or (S2) and are hence in the situation described in Lemmas \ref{lem:spec:diff:sides}, \ref{lem:spec:same:sides} and Remarks \ref{rem:S1}, \ref{rem:S2}. 
 
 \emph{Counting intersection points between $d_5$ and $b$.}
 If we follow the proof of Proposition \ref{Prop:main:g3}, then the first place it can fail in cases (S1) and (S2) is the proof of Lemma \ref{lem:one:new:inters}. However, as we will show below, even the stronger statement of Lemma \ref{Lem:two:NewVer} holds true in cases (S1) and (S2). The possible failure of the argument in Lemma \ref{Lem:two:NewVer} follows from the fact that an arc $s'$ of $c\cap N_a$ of type A, B, or C may exist, which yields an arc $s$ of $d_5\cap N_a$ that is parallel to $a$ in fewer than two rectangles of $N_{a\cap b}$. Moreover, the proof of that lemma provides 
 a specific description of possible configurations of $p,q$ and $s'$ that may lead to such a failure. Three possibilities exist in the case (S1)/(S2): $s'$ may be an arc of type A in $r_A/r_B$, $s'$ may be an arc of type B that connects $r_A$ and $r_B$, or $s'$ may be an arc of type C that connects $r_A$ and $r_B$. The following lemma shows that in each of these cases, strong restrictions are given for reductions of types IIIb/IIIc.
  \begin{lem}\label{Lem:s13:inner}
 Let $c$ be as in the proof of Proposition \ref{Prop:main:g3}.
  \begin{enumerate}
   \item Let $s$ be an arc of $d_3\cap N_a$ that corresponds to an arc $s'$ of $c\cap N_a$. If $s'$ is an arc of type C with endpoints in $r_A$ and $r_B$, then $s$ can admit a reduction of type IIIc only on one side of $s'$.
 \item Let $s$ be an arc of $c$, which is the innermost long bounding segment for $\Delta$ (that is, $s$ is a bounding segment for $\Delta$ in rectangles of $N_{(a\cup b)\bez(a\cap b)}$) and which leads to a reduction of type IIIb across $\Delta$. Then, the arcs of $c\cap N_a$ that precede/follow $s$ are arcs of type B.
 \item Suppose that a long bounding segment of $c$ exists, which leads to a reduction of type IIIb across $\Delta$. Then, no arc of $c$ which is of type $A$ in $r_A$ or $r_B$ leads to a reduction of type IIIb across $\Delta$.
 \item No arc of $c\cap N_a$ that is of type A in $r_A$ or $r_B$ leads to a reduction of type IIIb across $\Delta$.
  \item If $s_1$ is an arc of $c$ of type B with endpoints in $r_A$ and $r_B$, then $s_1$ can lead to a reduction of type IIIb only on one side of $s_1$. 
  \end{enumerate}
 \end{lem}
\begin{proof}\noindent
\begin{enumerate}
 \item If $s'$ is an arc of type C, then by Lemma \ref{LemSpecS1c}, at least one of the arcs that precede/follow $s'$ cannot lead to a reduction of type IIIc (because it turns to the left as it enters $N_a$).
 \item By Lemma \ref{LemSpecS1a}, the arcs that precede/follow $s$ cannot be both of type A, and by Lemma \ref{LemSpecS1ab}, they are not arcs of types A and B. Hence, both arcs must be of type B.
 \item If we choose the arc $s$ of $c$ that admits a reduction of type IIIb across $\Delta$ and is the innermost long bounding segment for $\Delta$, then by the previous point, we know that the arcs $s_1,s_2$ of $c$ that precede/follow $s$ are of type B in $r_A$ and $r_B$ (Figure \ref{r23}(ii)). But then, by Lemma \ref{LemSpecS1b}, we know that the arcs that respectively precede $s_1$ and follow $s_2$ do not admit reductions of type IIIb. Hence, they are the obstacles for the arcs of type A in $r_A,r_B$ to admit such a reduction.
 \item The conclusion is a direct consequence of the previous point.
 \item At this point, we know that if an arc $s$ of $c$ admits a reduction of type IIIb across $\Delta$, then this must be a long bounding segment that connects two arcs $s_1,s_2$ of $c\cap N_a$ of type B. But then, by Lemma \ref{LemSpecS1b}, the arc that precedes $s_1$, which is different from $\pm s$, does not admit reductions of type IIIb.
\end{enumerate}
\end{proof}
The above lemma implies that even in cases (S1)/(S2), none of the arcs $s'$ of $c\cap N_a$ can be reduced so that the resulting arc $s$ of $d_5\cap N_a$ is parallel to $a$ in fewer than 2 rectangles of $N_{a\cap b}$. Hence, Lemma \ref{Lem:two:NewVer} remains valid in these cases. In particular, $d_5$ is weakly rigid with respect to $a$ and $I(d_5,b)>I(d_5,a)$.

\emph{Weak winding of $d_5$.}
Finally, we need to show that $d_5$ winds weakly around $a$. Denote by 
$r_{p_1}=r_A,r_{p_2},\ldots, r_{p_n}$ and $r_{q_1}=r_B,r_{q_2},\ldots, r_{q_n}$ the rectangles of $N_{a\cap b}$, which contain the starting points of $p_1,p_2,\ldots,p_n$ and $-p_1,\ldots,-p_n$, respectively (Figure \ref{r21}). In case (S2), by $r_R$ denote the rectangle of $N_{a\cap b}$ that corresponds to $R$.

Let $s'$ be an arc of $c$ that is of type A in $r_{p_1}$, and let $s''$ and $s$ be the corresponding arcs of $d_3$ and $d_5$ respectively. By Lemma \ref{Lem:s13:inner}, the bottom part of $s''$ does not allow a reduction of type IIIb. The top part of $s''$ may admit reductions of type IIIc which can reach at most $r_{p_n}$. Hence $s$ is parallel to $a$ in all the rectangles $r_{q_1},\ldots, r_{q_n}$ and $r_R$ in case (S2). 
Moreover, 
if the reductions on the top part $s_t''$ of $s''$ reach $r_{p_n}$, then the corresponding arc $s_t$ of $d_5$ is two-sided with respect to $b$, that is, $s_t$ together with an arc of $b$ that connects the intersection points of $s_t$ and $b$ is a two-sided circle [Figure \ref{r24}(ii)]. Hence, 
$s_t$ is 
not a one-sided boundary 3-segment of an exterior hexagon and by Lemma \ref{lem:spec:same:side:prod:weak}, $s$ is weakly rigid with respect to $a$ on both sides of $r_{q_n}$ in case (S1) and on both sides of $r_R$ in case (S2).
Moreover, $s$ is rigid with respect to $a$ on both sides of $r_{q_2},\ldots,r_{q_{n-1}}$.

In exactly the same way, by considering the arc of $c$ which is of type A in $r_{q_1}$, 
we prove that $d_5$ is weakly rigid with respect to $a$ on both sides of rectangles $r_{p_2},\ldots,r_{p_{n}}$.
Hence, to finish the proof, we need to show the existence of an arc of $d_5\cap N_a$ that is  weakly rigid on both sides of $r_{q_1}$ and an arc of $d_5\cap N_a$ that is weakly rigid on both sides of $r_{p_1}$.

In case (S1), let $s'$ be an arc of $c$ that is of type A in $r_{p_n}$ and that is weakly rigid on both sides of $r_{p_n}$.
Let $s''$ and $s$ be the corresponding arcs of $d_3$ and $d_5$ respectively. By Lemma \ref{s3:lem1}, the bottom part of $s''$ does not allow a reduction of type IIIb. It may admit reductions of type IIIc, but by Lemma \ref{rem:no:adjacency:s1}, they can not reach $r_{p_1}$. The top part of $s''$ does not admit a reduction of type IIIb (Lemma \ref{s3:lem1}) and it does not allow a reduction of type IIIc (Proposition \ref{prop:prop:segments} and Remark \ref{Uw:Red:join:cascade}). Hence, $s$ is parallel to $a$ in $r_{p_1},r_{q_1},r_{q_2}$. Similarly, by taking an arc $s'$ of $c$ that is weakly rigid on both sides of $r_{q_1}$, we construct an arc $s$ of $d_5$ that is parallel to $a$ in $r_{p_2},r_{p_1},r_{q_1}$.




In case (S2), by Lemma \ref{rem:no:adjacency:s2}, we know that either the segments that go down from the starting points of $p_1,\ldots,p_n$ or the segments that go up from the terminal points of $p_1,\ldots,p_n$ are not mutually joinable. The argument in both cases is completely analogous. Hence, the latter case is assumed. Let $s'$ be an arc of $c$ that is of type A in $r_{q_n}$ and is weakly rigid on both sides of $r_{q_n}$. 
Let $s''$ and $s$ be the corresponding arcs of $d_3$ and $d_5$ respectively. By Lemma \ref{s3:lem1}, the top part of $s''$ does not allow a reduction of type IIIb. It may admit reductions of type IIIc, but by our assumption they cannot reach $r_{q_1}$. The bottom part of $s''$ does not admit a reduction of type IIIb. Hence, $s$ is parallel to $a$ in $r_{p_2},r_{p_1},r_{q_1}$. Similarly, if $s''$ is an arc of $c$ that is of type $A$ in $r_{p_n}$ 
and $s$ is 
the corresponding arc of 
$d_5$,
then $s$ is parallel to $a$ in $r_{q_1}, r_{q_2}$, and is either parallel to $a$ in $r_{p_1}$, or is weakly rigid on the $r_{p_1}$ side of $r_{q_1}$ (Lemma \ref{lem:spec:same:side:prod:weak}). This completes the proof that $d_5$ winds weakly around $a$.

\end{proof}
 \section{The special case (S3)}\label{sec:specB}
\begin{lem}\label{lem:spec:pos:s3}
 Let $a$ and $b$ be two generic two-sided circles in $N$ such that $I(a,b)\geq 4$, and assume that 
 oriented segments $p,q$ of $b$ and a double segment $R$ of $b$ exist such that 
 $p$ starts and terminates on one side of $a$, $q$ starts and terminates on the other side of $a$,
 each double segment of $b$, which is different from $R$ contains an oriented segment joinable to $p$ or $q$, $\{p,q,R\}$ is positively oriented [that is, $a$ and $b$ are in the special position (S3)]. 
 Then, the oriented arcs that constitute $R$ are joinable neither to $p$ nor to $q$.
\end{lem}
\begin{proof}
Suppose that in each double segment of $b$, an oriented segment joinable either to $p$ or $q$ exists. By Proposition \ref{prop:prop:segments}, all segments that start at endpoints of segments joinable to $p$ and are not joinable to $-p$ are joinable to $q$. And vice versa, all segments that start at endpoints of segments joinable to $q$ and are not joinable to $-q$ are joinable to $p$. Hence, $b$ is a circle on an annulus [which is an union of adjacency disks between the segments joinable to $p$ and $q$ -- see Figure \ref{r15}(ii)]. However, this contradicts the assumption that $I(a,b)\geq 4$.
\end{proof}
\begin{uw}\label{rem:s3:conf}
 Let $a$ and $b$ be circles in the special position (S3), let $p_1,p_2,\ldots,p_n$ be all oriented segments of $b$ joinable to $p$, and let $q_1,\ldots,q_m$ be all oriented segments joinable to $q$. The union of adjacency disks between $p_1,p_2,\ldots,p_n$ gives a rectangle $\Delta_p$ and the union of adjacency disks between $q_1,q_2,\ldots,q_m$ gives a rectangle $\Delta_q$. By~Lemma \ref{lem:spec:pos:s3}, we know that one of the segments $p_1,\ldots,p_n,q_1,\ldots,q_m$ ends in $R$ and is followed by an arc that is joinable neither to $p$ nor $q$. Without loss of generality we can assume that $p$ is above $a$, $p_n$ ends in $R$, $p_1,p_n$ are on the boundary of $\Delta_p$ and $q_1, q_m$ are on the boundary of $\Delta_q$. Given that all oriented segments $p_1,\ldots,p_{n-1}$ must be followed by segments joinable to $q$, and all segments $q_1,\ldots,q_m$ must be followed by segments joinable to $p$, we can also assume that $p_1,\ldots,p_{n-1}$ are followed by respectively $q_1,\ldots,q_m$. Hence the configuration is as in Figure \ref{r21b}. 
 \begin{figure}[h]
 \begin{center}
\fig{0.72}{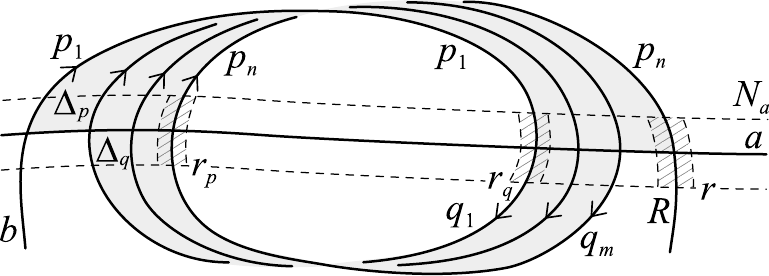}
 \caption{Configuration of arcs in case (S3).}\label{r21b} %
 \end{center}
 \end{figure}
 (Note that the mutual position of $p,q$ and $R$ is determined by the assumption that $\{p,q,R\}$ is positively oriented.) 
\end{uw}
For the rest of this section we will use the notation introduced in the above remark. Moreover, by $r,r_p,r_q$ we denote the rectangles of $N_{a\cap b}$ that correspond respectively to $R$, the initial point of $p_n$, and the initial point of $q_1$ (Figure \ref{r21b}).
\begin{lem}\label{lem:spec:same:side:hex}
 No component of $N\bez N_{a\cup b}$ is an exterior hexagon.
\end{lem}
\begin{proof}
 Given that the segment that connects the terminal point of $p_n$ with the initial point of $p_1$ is one-sided, checking that $r$ can be a vertex of an exterior $n$-gon only for $n=4,8$ is straightforward. All the other exterior $n$-gons are rectangles.
\end{proof}
%
Let $\widehat{X}_a$ and $\widehat{X}_b$ be defined as in Section \ref{sec:weak}. 
\begin{prop}\label{Main:prop:s3}
 Let $a$ and $b$ be two generic two-sided circles in $N$ such that $I(a,b)\geq 4$. Then for any integer $k\neq 0$ we have
 \[t_a^k(\widehat{X}_b)\podz \widehat{X}_a\quad\text{and}\quad t_b^k(\widehat{X}_a)\podz \widehat{X}_b.\]
\end{prop}
\begin{proof}
 Observe first that if $a$ and $b$ are not in the special position (S3), then the proposition coincides with Proposition \ref{Main:prop:s1}. 
 Therefore, we can concentrate on the circles $a$ and $b$, which are in the special position (S3) and are hence
 in the situation described in Remark \ref{rem:s3:conf}.
 
 As in special cases (S1) and (S2), we observe that weak winding is sufficient for repeating most of the proof of Proposition \ref{Prop:main:g3}. Hence, we follow the lines of the proof of Proposition \ref{Main:prop:s1} and we concentrate on places where that proof can fail. The first such place is the proof that $d_5$ is weakly rigid with respect to $a$. In the proof of Proposition \ref{Main:prop:s1} we obtained weak rigidity of $d_5$ as a consequence of Lemma \ref{Lem:two:NewVer}, which may not be true in the case (S3). 
 However, as a replacement we have the following lemma:
 \begin{lem}\label{lem:two:intersections:s3}
 Let $s'$ be an arc of $c\cap N_a$ of type A, B or C, and let $s$ be the arc of $d_5\cap N_a$ which corresponds to $s'$.
 \begin{enumerate}
  \item If $s'$ is an arc of type A in $r_p$ or $s'$ is an arc of type C that connects $r_p$ and $r_q$, then $s$ is parallel to $a$ in $r$.
  \item If $I(a,b)=4$ and $s'$ is an arc of type A in $r$ or $s'$ is an arc of type C with the top part in $r$, then $s$ is parallel to $a$ in $r_p$.
  \item If $s'$ is not as in previous points, then $s$ is parallel to $a$ in at least two rectangles of $N_{a\cap b}$.
 \end{enumerate}
 \end{lem}
 \begin{proof}
 Observe first that the boundary component of $N_{a\cup b}$ that contains $p_n$ and $p_1$ cannot bound an exterior rectangle. Otherwise, $a$ would bound a \Mob. Hence the segments of $c$, which run parallel to $p_1$, do not lead to reductions of types IIIb nor IIIc. Let $s''$ be an arc of $d_3$ that corresponds to $s'$.
 \begin{enumerate}
  \item If $s'$ is an arc of type A in $r_p$, then the bottom part of $s''$ can admit a reduction of IIIb and then both sides of the obtained arc may admit reductions of type IIIc. By Lemma \ref{lem:spec:pos:s3}, these reductions cannot reach $R$.
  Hence, $s$ is parallel to $a$ in $r$. 

  The situation is completely analogous if $s'$ is an arc of type C that connects $r_p$ and $r_q$ (by Lemma \ref{lem:ends:ab}, $s''$ does not admit reductions of type IIIb in this case).
  \item If $s'$ is an arc of type A in $r$, then the top part of $s''$ may admit some reductions of type IIIc, which can reach at most $r_{q}$. The bottom part of $s''$ may admit a reduction of type IIIb, but then it does 
  not admit any further reductions of type IIIc. Hence, $s$ is parallel to $a$ in $r_p$.
  
  The situation is completely analogous if $s'$ is an arc of type C with the top part in $r$ (by Lemma \ref{lem:ends:ab}, $s''$ does not admit reductions of type IIIb in this case).
  \item If $s'$ is not as in the previous points, then it is straightforward to check that $s$ is parallel to $a$ in either the rectangles of $N_{a\cap b}$ that contain the initial points of $p_{n-1}$ and $p_n$, or in rectangles of $N_{a\cap b}$ that contain the terminal points of $p_1$ and $p_2$.  
 \end{enumerate}
 \end{proof}
 As a consequence of the above lemma, we obtain that every arc $s$ of $d_5$ that is parallel to $a$ in a rectangle $t$ of $N_{a\cap b}$ is weakly rigid on one side of $t$. In fact, if $s$ is parallel to $a$ in at least two rectangles of $N_{a\cap b}$, then $s$ is rigid in $t$. If $s$ is parallel to $a$ in only one rectangle of $N_{a\cap b}$, then $s$ was obtained from an arc of $d_3$ by reductions of types IIIb and IIIc.  In such a case, by 
  Lemmas \ref{lem:spec:same:side:prod:weak} and \ref{lem:spec:same:side:hex}, $s$ is weakly rigid on one side of $t$. In particular, $d_5$ is weakly rigid with respect to $a$.
 
 Finally, we will show that $d_5$ winds weakly around $a$. Let $s_1'$ and $s_2'$ be arcs of $c$ of type A in $r_q$ and the rectangle $r_{p_1}$ of $N_{a\cap b}$ that contain the initial point of $p_1$, respectively. Assume also that $s_1'$ and $s_2'$ are weakly rigid below $a$, and let $s_1'',s_2''$ and $s_1,s_2$ be arcs of $d_3$ and $d_5$ that correspond to $s_1',s_2'$, respectively. The arc $s_2''$ does not admit any reductions of types IIIb and IIIc. Hence, $s_2$ is parallel to $a$ in all rectangles of $N_{a\cap b}$ except $r_{p_1}$. The bottom part of $s_1''$ may admit some reductions of type IIIc, but these reductions cannot reach $r$. Hence, $s_1$ is parallel to $a$ in $r$ and in all the rectangles of $N_{a\cap b}$ between $r_{p_1}$ and $r_p$. This proves that $d_5$ winds weakly around $a$. 
%
%
 \end{proof}
 \section{The case of $I(a,b)=3$ with nonorientable $N_{a\cup b}$} \label{sec:g3}
If $I(a,b)=3$, then we still follow the proofs of Propositions~\ref{Prop:main:g3} and \ref{Main:prop:s1}. However, as in the special case (S3), the problem is that in general, if $I(a,b)=3$, then Lemma \ref{Lem:two:NewVer} is not true. Fortunately, this case is quite special because of the following proposition:
\begin{prop}\label{Prop:gen:nodisk:i3}
 If $a$ and $b$ are two generic two-sided circles in a surface $N$ such that $|a\cap b|=I(a,b)=3$ and $N_{a\cup b}$ is nonorientable,
 then the exterior $n$-gons for $N_{a\cup b}$ can exist only if $n=10$ or $n=12$.
\end{prop}
\begin{proof}
 If all segments of $b$ are two-sided, then $N_{a\cup b}$ is orientable. Hence, one-sided segments of $b$ exist, and two such segments $p_1,p_2$ must exist. If we denote the two-sided segment of $b$ as $p_3$, then we can assume that $p_1,p_2,p_3$ are oriented so that $p_2$ follows $p_1$ and $p_3$ follows $p_2$.
 
 If all segments $p_1,p_2,p_3$ start and terminate on different sides of $a$, then the arc of $a$ that connects the initial point of $p_1$ with the initial point of $p_2$ starts and terminates on the same side of $b$. Hence we can interchange $a$ with $b$, and we can always assume that at least one segment of $b$ starts and terminates on the same side of $a$. In such a case one of the remaining segments must also start and terminate on the same side of $a$, and the final segment must connect two different sides of $a$. Hence we have the configuration of arcs as in Figure \ref{r03b}. 
 \begin{figure}[h]
 \begin{center}
\fig{0.99}{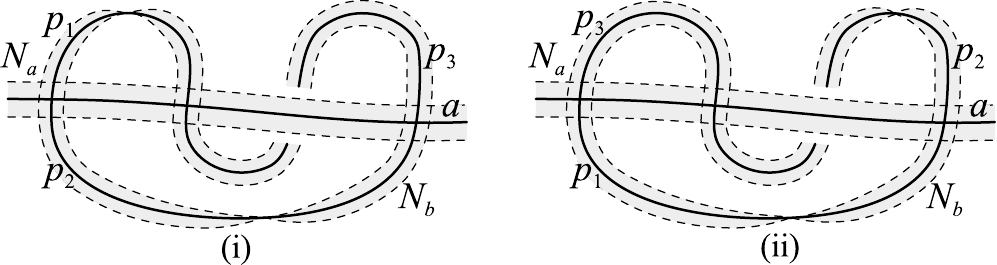}
 \caption{Configurations of segments of $b$, Proposition \ref{Prop:gen:nodisk:i3}.}\label{r03b} %
 \end{center}
 \end{figure}
 We still have two possibilities: either $p_3$ is a segment that connects two different sides of $a$ [Figure \ref{r03b}(i)], or $p_3$ starts and terminates on the same side of $a$ [Figure \ref{r03b}(ii)]. Checking that in the first case the boundary of $N_{a\cup b}$ is connected and it is a 12-gon is straightforward. In the second case, the boundary of $N_{a\cup b}$ has two components: a bigon and a 10-gon. 
\end{proof}
The above proposition implies that if $I(a,b)=3$, then no adjacent segments of $b$ exist. Hence, no reductions of type IIIc exist. Moreover, the notion of weakly rigid arcs simplifies, given that no exterior hexagons exist. 


Let $\kre{X}_a$ be the set of isotopy classes of circles $c$ in $N$, which satisfy the following conditions:
\begin{enumerate}
 \item $c\in {\cal C}$,
 \item $I(c,a)=|c\cap a|$, $I(c,b)=|c\cap b|$,
 \item $I(c,a)<I(c,b)$,
 \item $c$ winds weakly around $a$.
\end{enumerate}
Similarly, we define $\kre{X}_b$ by requiring (1)--(2) above and additionally
\begin{enumerate}
 \item[(3')] $I(c,b)<I(c,a)$,
 \item[(4')] $c$ winds weakly around $b$.
\end{enumerate}
The main difference between sets $\kre{X}_b$ and $\widehat{X}_b$ is the lack of the assumption that $c$ is weakly rigid with respect to $b$. As a replacement for this assumption we have the following lemma:
\begin{lem}\label{lem:g3:rig:one:side}
 Let $a$ and $b$ be two generic two-sided circles in $N$ such that $I(a,b)=3$ and $c\in\kre{X}_b$. Let $s$ be an arc of $c\cap N_a$ of type A and let $s'$ be the arc of $d_3$ that corresponds to $s$. Then, $s'$ can admit a reduction of type IIIb with only one orientation of $s$.
\end{lem}
\begin{proof}
 Suppose to the contrary that $s$ leads to reductions of type IIIb with both orientations of $s$ (that is, on both sides of $a$) and let $r$ be the rectangle of $N_{a\cap b}$ that contains $s$. Observe first that both ends of $s'$ must be involved in two different reductions of type IIIb; otherwise, $s$ would intersect $a$ only once, which would contradict the assumption that $c$ winds around $b$. By Proposition \ref{Prop:gen:nodisk:i3}, at most one exterior $n$-gon $\Delta$ exists. Hence, both reductions on $s'$ must correspond to the same one-sided segment $p$ of $b$; this segment corresponds to one side of the bigon defining a reduction. Therefore, if we follow $s$ in both directions until we obtain the arcs $s_1,s_2$ of $c\cap N_a$ which intersect $a$, then $s_1$ and $s_2$ must enter $N_a$ in $r$. Moreover, the arcs of $c$ that connect $s$ with $s_1$ and $s_2$ run through rectangles of $N_{a\cup b}$, which together with $r$, constitute a \Mob\ strip $M$. Hence, $s_1$ and $s_2$ enter $r$ on the same side of $s$ [Figure \ref{r07b}(i)]. 
 \begin{figure}[h]
 \begin{center}
\fig{0.9}{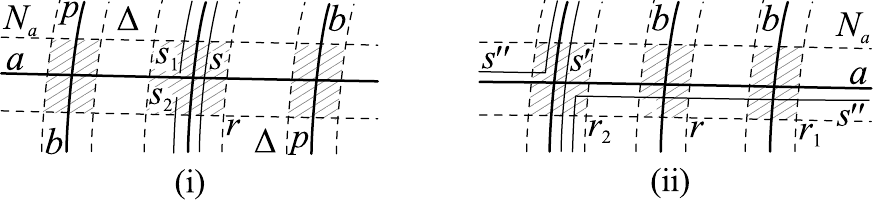}
 \caption{Configurations of segments of $c$ and $d_3$, Lemma \ref{lem:g3:rig:one:side} and Proposition \ref{Main:prop:i3}.}\label{r07b} %
 \end{center}
 \end{figure}
 Therefore, by Lemma \ref{lem:ends:ab} and Remark \ref{Rem:weak:subs2}, at least one of the arcs $s_1$ and $s_2$ must be an arc of type A, which leads to a contradiction with Lemma~\ref{MobRest}.
\end{proof}
\begin{prop}\label{Main:prop:i3}
 Let $a$ and $b$ be two generic two-sided circles in $N$ such that $I(a,b)=3$. Then for any integer $k\neq 0$, we have
 \[t_a^k(\kre{X}_b)\podz \kre{X}_a\quad\text{and}\quad t_b^k(\kre{X}_a)\podz \kre{X}_b.\]
\end{prop}
\begin{proof}
 As in the special case (S3), the main problem that may lead to the failure of the proof of Proposition \ref{Main:prop:s1} is the fact that if $I(a,b)=3$, then Lemma \ref{Lem:two:NewVer} may not be true.
 However, as a replacement for that lemma we have the following slightly weaker result:
\begin{lem}\label{lem:two:intersections:i3}
Let $s'$ be an arc of $c\cap N_a$ of type A, B or C, and let $s$ be the arc of $d_5\cap N_a$ that corresponds to $s'$. Then, $s$ is parallel to $a$ in at least one rectangle of $N_{a\cap b}$
 \end{lem}
 \begin{proof}
 Let $s''$ be an arc of $d_3\cap N_a$ that corresponds to $s'$. If $s'$ is an arc of type A, then by Lemma \ref{lem:g3:rig:one:side}, $s''$ may admit a reduction of type IIIb only on one side of $a$. Hence, $s$ is parallel to $a$ in at least one rectangle of $N_{a\cap b}$. If $s'$ is an arc of type B, then $s''$ may admit reductions of type IIIb on both sides of $a$, but this leads to same conclusion as above. Arcs of type C does not allow reductions of type IIIb (see Lemma \ref{lem:ends:ab} and Remark \ref{Rem:weak:subs2}). Hence, in this case, $s$ is parallel to $a$ in exactly one rectangle of $N_{a\cap b}$.
 \end{proof}
 As a consequence of the above lemma and Proposition \ref{Prop:gen:nodisk:i3} we have the following simplified version of Lemma~\ref{lem:spec:same:side:prod:weak}:
\begin{lem}\label{lem:g3:prod:weak}
 Let $q$ be an oriented arc of $d_5$ that starts in a rectangle $t$ of $N_{a\cap b}$ as an arc parallel to $a$, then it follows $a$ to the next rectangle of $N_{a\cap b}$, and then it follows an arc $s$ obtained from an arc of $d_3\bez N_a$ by a reduction of type IIIb. Then, $q$ is weakly rigid in $t$ with respect to $a$.\qed
\end{lem}
\emph{Weak rigidity of $d_5$.}
 Fix a rectangle $r$ in $N_{a\cap b}$ and assume that the orientation of arcs parallel to $a$ in $r$ is such that these arcs point to the rectangle $r_1$, which is on the right of $r$ [if the orientation is opposite we can rotate the whole picture by $180^\circ$; see Figure \ref{r07b}(ii)]. Let $s'$ be an arc of $c\cap N_a$ of type A in a rectangle $r_2$ of $N_{a\cap b}$ different from $r$ and $r_1$, and assume that $s'$ is weakly rigid in $r_2$ with respect to $b$ and with the orientation pointing down (we use the assumption that $c$ winds weakly around $b$). The arc $s''$ of $d_3$ that corresponds to $s'$ is parallel to $a$ in $r$ and $r_1$. If this arc does not admit a reduction of type IIIb, then $s'$ is in fact an arc of $d_5$ and this arc is rigid in $r$. If, on the other hand, $s''$ admit a reduction of type IIIb, then this reduction must be on the top part of $s''$. Hence, by Lemma \ref{lem:g3:prod:weak}, the arc $s$ of $d_5$ that corresponds to $s''$ is weakly rigid in $r$.
  
  \emph{Counting intersection points between $d_5$ and $b$.} 
 Lemma \ref{lem:two:intersections:i3} guarantees that for each intersection point of $c$ and $a$, we have at least one intersection point of $d_5$ and $b$. To prove that $I(d_5,b)>I(c,a)$, we need to show that for some arcs of $c\cap N_a$ the corresponding arcs of $d_5\cap N_a$ intersect $b$ at least twice. 
 
 In fact, by Proposition \ref{Prop:gen:nodisk:i3}, at most one exterior $n$-gon exists. Hence, all reductions of type IIIb must correspond to the same segment $p$ of $b$; this segment corresponds to one side of the bigon that defines a reduction. Let $r_A$ and $r_B$ be the rectangles of $N_{a\cap b}$ that contain the endpoints $A$ and $B$ of $p$. Let $r_1$ and $r_2$ be rectangles of $N_{a\cap b}$ that precede $r_A$ and $r_B$, respectively, with respect to the twisting direction in $N_a$, that is $r_1$ is the right [left] neighbor of $r_A$ if $p$ approaches $a$ from above [from below], similarly for $r_2$ (see Figure \ref{r07c}). 
 \begin{figure}[h]
 \begin{center}
\fig{0.6}{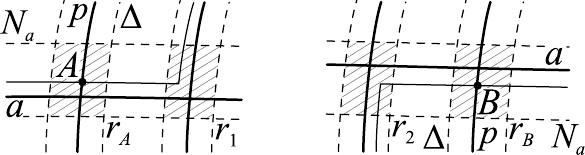}
 \caption{Configuration of rectangles of $N_{a\cap b}$, Proposition \ref{Main:prop:i3}.}\label{r07c} %
 \end{center}
 \end{figure}
 All arcs of $c\cap N_a$ of type A that lead to the reduction of type IIIb must be contained in either $r_1$ or $r_2$. Hence, if we choose an arc $s'$ of $c\cap N_a$ that is of type A in a rectangle $r$ of $N_{a\cap b}$ different from $r_1$ and $r_2$, then the corresponding arc $s''$ of $d_3$ does not admit reductions of type IIIb. Therefore, the arc $s$ of $d_5\cap N_a$ that corresponds to $s'$ is parallel to $a$ in two rectangles of $N_{a\cap b}$.
\end{proof}
 \section{The case of $I(a,b)=2$ with nonorientable $N_{a\cup b}$}\label{sec:2}
Checking that Propositions \ref{Prop:main:g3} and \ref{Main:prop:s3} are false if $I(a,b)=2$ 
and $N_{a\cup b}$ is nonorientable is not difficult. Hence, we need slightly more sophisticated analysis in that case.

The case in question is special because of the following proposition:
\begin{prop}\label{Prop:gen:nodisk}
 If $a$ and $b$ are two generic two-sided circles in a surface $N$ such that $|a\cap b|=I(a,b)=2$ and $N_{a\cup b}$ is nonorientable,
 then no component of $N\bez (a\cup b)$ is a disk.
\end{prop}
\begin{proof}
Observe that $N_{a\cup b}$ is a Klein bottle with two boundary components (Figure \ref{r20}(i)).
\begin{figure}[h]
 \begin{center}
\fig{0.71}{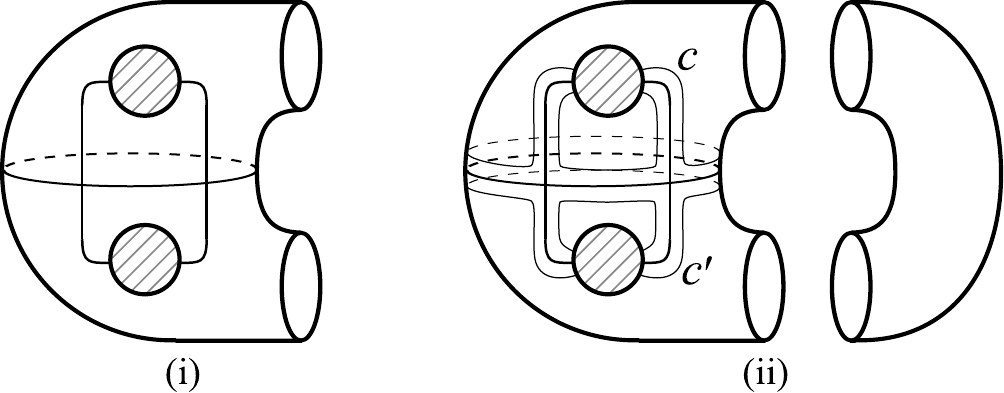}
 \caption{Klein bottle with two holes as a regular neighborhood of $a\cup b$.}\label{r20} %
 \end{center}
 \end{figure}
Hence if
one of the components of $N\bez(a\cup b)$ is a disk, then one of the circles $a$ or $b$ bounds a
\Mob\ which is a contradiction.
\end{proof}
For a circle $c\in{\cal C}$ define
\[\begin{aligned}
 J(c,a)&=\text{number of connected components of $c\bez N_a$}\\
 J(c,b)&=\text{number of connected components of $c\bez N_b$}.
\end{aligned}\]
\begin{prop} \label{Prop:gen2:J:invar}
 Let $a$ and $b$ be two generic two-sided circles in a surface $N$ such that $|a\cap b|=I(a,b)=2$ and $N_{a\cup b}$ is nonorientable. If $c,c'\in{\cal C}$ such that $c$ is isotopic to $c'$, then $J(c,a)=J(c',a)$ and $J(c,b)=J(c',b)$.
\end{prop}
\begin{proof}
Suppose first that $|c\cap c'|>0$ and let $\Delta$ be a bigon formed by $c$ and $c'$. By taking the inner-most bigon, we can assume that the interior of $\Delta$ is disjoint from $c\cup c'$. Given that the boundary of $N_{a\cup b}$ is 
disjoint from $c\cup c'$, if a component of $N\bez N_{a\cup b}$ intersects $\Delta$, then this component must be a disk.
By Proposition \ref{Prop:gen:nodisk}, this case is not possible. Hence, $\Delta$ is contained in $N_{a\cup b}$. Moreover, we can assume
that the vertices of $\Delta$ are in the interior of rectangles of $N_{a\cup b}$. 

Fix a rectangle $r$ in $N_{a\bez b}\cup N_{b\bez a}$, and then let $\Delta_r$ be a connected component of $\Delta\cap r$. 
Given that $\Delta\subset N_{a\cup b}$ and $c,c'$ do not turn back in any of the rectangles of $N_{a\cup b}$, $\Delta_r$ must be
either a rectangle, a triangle, or a bigon with two sides being arcs of $c$ and $c'$ [Figure \ref{r18a}].
\begin{figure}[h]
 \begin{center}
\fig{0.94}{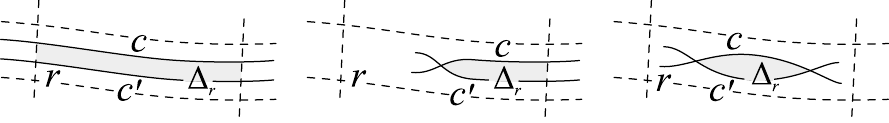}
 \caption{Arcs of $c$ and $c'$ in a rectangle $r$.}\label{r18a} %
 \end{center}
 \end{figure}
In any case, if we remove the bigon 
$\Delta$, that is, if we replace $c'$ with the circle $c''$  isotopic to $c'$ which is obtained by pushing $c'$ across $\Delta$ (Figure \ref{r18}(i)), then 
\[J(c'',a)=J(c',a)\text{ and }J(c'',b)=J(c',b).\]
\begin{figure}[h]
 \begin{center}
\fig{0.67}{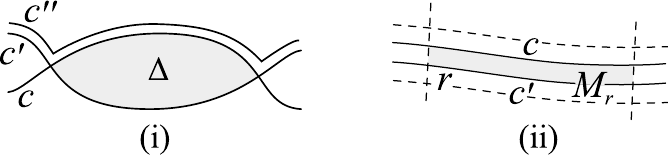}
 \caption{Arcs of $c$ and $c'$ in a rectangle $r$.}\label{r18} %
 \end{center}
 \end{figure}
Therefore, we can assume that $c$ and $c'$ are disjoint. This assumption means that an annulus $M$ in $N$ with boundary curves $c$ and $c'$ exists. 

If a component of $N\bez N_{a\cup b}$ intersects $M$, then this component must be an annulus with boundary curves
isotopic to $c$ and $c'$. In such a case $N$ is a nonorientable surface of genus 4 [Figure \ref{r20}(ii)] and
\[J(c,a)=J(c',a)=2\text{ and }J(c,b)=J(c',b)=2.\]
Finally, assume that $M$ is contained in $N_{a\cup b}$. As in the case of a bigon formed by $c$ and $c'$, if $r$ is
a rectangle in $N_{a\bez b}\cup N_{b\bez a}$ and $M_r$ is a connected component of $M\cap r$, then $M_r$
must be a rectangle with two sides being arcs of $c$ and 
$c'$ that connect opposite sides of $r$ [Figure~\ref{r18}(ii)].
This implies that $M_r$ gives in $r$ exactly one arc of $c$ and one arc of $c'$. This result means that 
$J(c,a)=J(c',a)$ and $J(c,b)=J(c',b)$.
\end{proof}
For two generic two-sided circles $a,b$ in $N$ such that $|a\cap b|=I(a,b)=2$, we define $\fal{X}_a$ as the set of isotopy classes 
of circles in $N$, which satisfy the following conditions:
\begin{enumerate}
 \item $c\in {\cal C}$,
 \item $J(c,a)<J(c,b)$,
 \item $c$ winds around $a$.
\end{enumerate}
Similarly, we define $\fal{X}_b$ by requiring (1) above and additionally
\begin{enumerate}
 \item[(2')] $J(c,b)<J(c,a)$,
 \item[(3')] $c$ winds around $b$.
\end{enumerate}

As an analog of Proposition \ref{Prop:main:g3}, we have
\begin{prop}\label{Prop:gen2:main}
 Let $a$ and $b$ be two generic circles in $N$ such that $I(a,b)=2$ and $N_{a\cup b}$ is nonorientable. Then for any integer $k\neq 0$, we have
 \[t_a^k(\fal{X}_b)\podz \fal{X}_a\quad\text{and}\quad t_b^k(\fal{X}_a)\podz \fal{X}_b.\]
\end{prop}
\begin{proof}
 As in the proof of Proposition \ref{Prop:main:g3}, we concentrate on the inclusion $t_a^k(\fal{X}_b)\podz \fal{X}_a$.
 We begin by constructing the circle $d=t_a^k(c)$ and as before, we assume that $t_a^k$ twists to the right in $N_a$. We perform reductions of type I on $d$, and as a result, we obtain a circle 
 $d_1\in{\cal C}$ which winds around $a$. Observe that by Proposition \ref{Prop:gen2:J:invar}, showing that $J(d_1,b)>J(d_1,a)$ is sufficient (we do not need to focus on reductions of types II--III). 
 
 Let $n_A, n_B, n_C, n_D$ be numbers of arcs of $c\cap N_a$ of types A, B, C, D, respectively (Figure \ref{r13}). In particular,
 \[J(d_1,a)=n_A+n_B+n_C+n_D.\]
 To determine the number $J(d_1,b)$, suppose first that $|k|=1$. Each arc of $c\cap N_a$ of type A gives an
 arc of $d_1$ which goes once around $a$ and therefore gives $I(a,b)=2$ in $J(d_1,b)$. An arc of $c\cap N_a$ of type B gives
 $I(a,b)+1=3$ in $J(d_1,b)$, and an arc of type C gives $I(a,b)-1=1$. An arc of $c\cap N_a$ of type D does not change after the twist and gives 1 in $J(d_1,b)$. Finally, if $|k|>1$, then for each arc of $c\cap N_a$ of types A--C, we have additional $(|k|-1)\cdot I(a,b)=2(|k|-1)$ arcs of $d_1\bez N_b$. Hence, we proved the following formula:
 \[J(d_1,b)=2n_A+3n_B+n_C+n_D+2(|k|-1)\cdot I(a,c).\]
 Given that $c$ winds around $b$, we have $n_A>0$. Hence, $J(d_1,b)>J(d_1,a)$.
\end{proof}
\begin{uw}
 The proof of Proposition \ref{Prop:gen2:main} can be repeated with minimal changes when $I(a,b)\geq 2$ and no 
 component 
 of $N\bez N_{a\cup b}$ is a disk or an annulus (for example if $N=N_{a\cup b})$. However, if disks are present in the 
 complement of $N_{a\cup b}$, then Proposition \ref{Prop:gen2:J:invar} is not true and the situation becomes difficult.
\end{uw}
\section{Twists generating a free group}\label{sec:thm}
Recall the so called 'Ping Pong Lemma' (see for example Lemma~3.15 of \cite{MargaliFarb}).
\begin{lem}\label{PingPong}
 Suppose that a group $G$ acts on a set $Y$, and $Y_1,Y_2\podz Y$ are nonempty and disjoint. Let
 $g_1,g_2\in G$ such that for every nonzero integer $k$,
 \[g_1^k(Y_2)\podz Y_1\quad\text{and}\quad g_2^k(Y_1)\podz Y_2.\]
 Then the group generated by $g_1$ and $g_2$ is a free group of rank 2.\qed
\end{lem}
\begin{tw}\label{Main:thrm}
 Let $a$ and $b$ be two generic two-sided circles in a nonorientable surface $N$. If $I(a,b)\geq 2$, then the group
 generated by $t_a$ and $t_b$ is isomorphic to the free group of rank 2. 
\end{tw}
\begin{proof}
If $I(a,b)\in\{2,3\}$ and $N_{a\cup b}$ is orientable, then we can repeat Ishida's proof \cite{Ishida} without any changes.

Let $(Y_a,Y_b)$ be equal to $(\widehat{X}_a,\widehat{X}_b)$ or $(\kre{X}_a,\kre{X}_b)$ or else $(\fal{X}_a,\fal{X}_b)$, where the sets $\widehat{X}_{a},\widehat{X}_{b},\kre{X}_{a},\kre{X}_{b},\fal{X}_{a},
\fal{X}_{b}$ are defined in Sections  \ref{sec:weak}, \ref{sec:g3}, and \ref{sec:2}. Observe that $Y_a$ and $Y_b$ satisfy
\[
   Y_a\cap Y_b=\emptyset,\ a\in Y_a,\ b\in Y_b.
\]
Hence, if $I(a,b)\in\{2,3\}$ and $N_{a\cup b}$ is nonorientable, or $I(a,b)\geq 4$, then the theorem follows from Lemma \ref{PingPong} and Propositions \ref{Main:prop:s3}, \ref{Main:prop:i3}, and \ref{Prop:gen2:main}.
\end{proof}

\bibliographystyle{abbrv}

\begin{thebibliography}{1}

\bibitem{Epstein}
D.~B.~A. {E}pstein.
\newblock Curves on 2--manifolds and isotopies.
\newblock {\em Acta Math.}, 115:83--107, 1966.

\bibitem{MargaliFarb}
B.~Farb and D.~Margalit.
\newblock {\em A Primer on Mapping Class Groups}, volume~49 of {\em Princeton
  Mathematical Series}.
\newblock Princeton Univ. Press, 2011.

\bibitem{Ishida}
A.~{I}shida.
\newblock The structure of subgroups of mapping class groups generated by two
  {D}ehn twists.
\newblock {\em Proc. Japan Acad. Ser. A Math. Sci.}, 72(10):240--241, 1996.

\bibitem{MargLein}
C.~J. {L}eininger and D.~Margalit.
\newblock Two--generator subgroups of the pure braid group.
\newblock {\em Geom. Dedicata}, 147(1):107--113, 2010.

\bibitem{Stukow_twist}
M.~Stukow.
\newblock Dehn twists on nonorientable surfaces.
\newblock {\em Fund. Math.}, 189:117--147, 2006.

\bibitem{StukowMargalit}
M.~{S}tukow.
\newblock Subgroups of the {T}orelli group generated by two symmetric bounding
  pair maps.
\newblock arXiv:1604.04833v1 [math.GT], 2016.

\end{thebibliography}

\end{document}